\documentclass{amsart}
\usepackage{a4wide}
\usepackage{graphics}
\usepackage{color}
\usepackage{amssymb}
\usepackage[all]{xy}
\usepackage{amsmath}
\usepackage{stmaryrd}
\usepackage{paralist}
\usepackage{color}
\CompileMatrices
\newtheorem*{introthm}{Theorem}

\newtheorem{theorem}{Theorem}[section]
\newtheorem{lemma}[theorem]{Lemma}
\newtheorem{proposition}[theorem]{Proposition}
\newtheorem{corollary}[theorem]{Corollary}

\theoremstyle{definition}

\newtheorem{definition}[theorem]{Definition}
\newtheorem{example}[theorem]{Example}

\newtheorem{construction}[theorem]{Construction}

\newtheorem{remark}[theorem]{Remark}
\theoremstyle{remark}

\numberwithin{equation}{theorem}
\def\Chi{{\mathbb X}}

\def\div{{\rm div}}

\def\quot{/\!\!/}
\def\mal{\! \cdot \!}

\def\rq#1{\widehat{#1}}
\def\t#1{\widetilde{#1}}
\def\b#1{\overline{#1}}
\def\bangle#1{\langle #1 \rangle}

\def\CC{{\mathbb C}}
\def\KK{{\mathbb K}}
\def\TT{{\mathbb T}}
\def\ZZ{{\mathbb Z}}
\def\RR{{\mathbb R}}

\def\NN{{\mathbb N}}
\def\QQ{{\mathbb Q}}
\def\PP{{\mathbb P}}
\def\GG{{\mathbb G}}

\def\WDiv{\operatorname{WDiv}}
\def\PDiv{\operatorname{PDiv}}
\def\id{{\rm id}}

\def\Lie{{\rm Lie}}

\def\CAut{\operatorname{CAut}}
\def\Aut{\operatorname{Aut}}
\def\Bir{\operatorname{Bir}}

\def\Cl{\operatorname{Cl}}

\def\LND{\operatorname{LND}}

\def\PASG{\operatorname{1-PASG}}

\def\SO{{\rm SO}}
\def\GL{{\rm GL}}

\def\Spec{{\rm Spec}}

\def\LND{{\rm LND}}

\def\cone{{\rm cone}}

\def\SL{{\rm SL}}
\def\Sp{{\rm Sp}}

\def\Aut{{\rm Aut}}
\def\rk{{\rm rk}}
\def\GL{{\rm GL}}

\begin{document}
\title[The automorphism group of a variety with torus action]%
{The automorphism group of a variety \\ with torus action of complexity one}
\author[I.~Arzhantsev]{Ivan Arzhantsev} 
\address{Department of Higher Algebra, 
Faculty of Mechanics and Mathematics, 
Moscow State University, 
Leninskie Gory 1, Moscow 119991, Russia}
\email{arjantse@mccme.ru}
\author[J.~Hausen]{J\"urgen Hausen} 
\address{Mathematisches Institut, Universit\"at T\"ubingen,
Auf der Morgenstelle 10, 72076 T\"ubingen, Germany}
\email{juergen.hausen@uni-tuebingen.de}
\author[E.~Herppich]{Elaine Herppich} 
\address{Mathematisches Institut, Universit\"at T\"ubingen,
Auf der Morgenstelle 10, 72076 T\"ubingen, Germany}
\email{elaine.herppich@uni-tuebingen.de}
\author[A.~Liendo]{Alvaro Liendo} 
\address{Mathematisches Institut, Universit\"at Bern,
Sidlerstrasse 5, 3012 Bern, Switzerland}
\email{alvaro.liendo@math.unibe.ch}

\begin{abstract}
We consider a normal complete rational variety 
with a torus action of complexity one.
In the main results, we determine the roots of 
the automorphism group and give an explicit 
description of the root system of its semisimple 
part.
The results are applied to the study of almost
homogeneous varieties.
For example, we describe all almost homogeneous 
(possibly singular) del Pezzo $\KK^*$-surfaces of 
Picard number one and all almost homogeneous 
(possibly singular) Fano threefolds of 
Picard number one having a reductive 
automorphism group with two-dimensional 
maximal torus.
\end{abstract}


\maketitle

\section{Introduction}

For any complete rational algebraic variety $X$, 
the unit component $\Aut(X)^0$ of its 
automorphism group is linear algebraic and 
it is a natural desire to understand the 
structure of this group.
Essential insight is provided by the roots,
i.e.~the eigenvalues of the adjoint
representation of a maximal torus on the 
Lie algebra.
Recall that if a linear algebraic group
is reductive, then its set of roots forms 
a so-called root system and, up to coverings, 
determines the group.
In the general case, the group is generated 
by its maximal torus and the additive 
one-parameter subgroups corresponding to the 
roots.
Seminal work on the structure of automorphism 
groups has been done by Demazure~\cite{Dem} 
for the case of smooth complete toric 
varieties $X$.
Here, the acting torus $T$ of $X$ is as well a 
maximal torus of $\Aut(X)^0$ and Demazure
described the roots of $\Aut(X)^0$ with 
respect to $T$ in terms of the defining fan 
of $X$; 
see~\cite{Co, BrGu, Mas, Nill} for further 
development in this direction.
Cox~\cite{Co} presented an approach to the 
automorphism group of a toric variety via 
the homogeneous coordinate ring and thereby 
generalized Demazure's results to the simplicial 
case; see~\cite{Br} for an application of
homogeneous coordinates to the study of 
automorphism groups in the more general case 
of spherical varieties.

In the present paper, we go beyond the toric case
in the sense that we consider normal complete 
rational varieties $X$ coming with an effective torus 
action $T \times X \to X$ of complexity one, 
i.e.~the dimension of~$T$ is one less than that 
of $X$; the simplest nontrivial examples are 
$\KK^*$-surfaces, see~\cite{OrWa1,OrWa2}.
Our approach is based on the Cox ring~$\mathcal{R}(X)$
and the starting point is the explicit description of 
$\mathcal{R}(X)$ in the complexity one case provided
by~\cite{HaHe,HaSu}; see also Section~\ref{sec:factgrad} 
for details.
Generators and relations of $\mathcal{R}(X)$ 
as well as the grading by the divisor class group $\Cl(X)$
can be encoded in a sequence $A = a_0,\ldots, a_r$ 
of pairwise linearly independent vectors in $\KK^2$ 
and an integral matrix
\begin{eqnarray*}
P
& = & 
\left[
\begin{array}{ccccc}
-l_0 & l_1 &   \ldots & 0 & 0  
\\
\vdots & \vdots   & \ddots & \vdots & \vdots
\\
-l_0 & 0 &\ldots  & l_{r} & 0
\\
d_0 & d_1 &\ldots  & d_{r} & d'
\end{array}
\right]
\end{eqnarray*}
of size $(n+m) \times (r+s)$, where $l_i$
are nonnegative integral vectors of length $n_i$,
the $d_i$ are $s \times n_i$ blocks,
$d'$ is an $s \times m$ block
and the columns of $P$ are pairwise different primitive 
vectors generating the column space $\QQ^{r+s}$ as a 
convex cone.
Conversely, the data $A,P$ always define 
a Cox ring $\mathcal{R}(X) = R(A,P)$ 
of a complexity one $T$-variety $X$.
The dimension of $X$ equals $s+1$ and the acting torus 
$T$ has $\ZZ^s$ as its character lattice.
The matrix $P$ determines the grading and the 
exponents occuring in the relations,
whereas $A$ is responsible for continuous 
aspects, i.e. coefficients in the relations.

The crucial concept for the investigation 
of the automorphism group $\Aut(X)$ are 
the {\em Demazure $P$-roots}, which we introduce 
in Definition~\ref{def:Pdemroot}.
Roughly speaking, these are finitely many integral 
linear forms $u$ on $\ZZ^{r+s}$ satisfying a couple 
of linear inequalities on the columns of $P$.
In particular, given $P$, the Demazure $P$-roots
can be easily determined.
In contrast with the toric case, the 
Demazure $P$-roots are divided into two types. 
Firstly, there are ``vertical'' ones corresponding 
to root subgroups whose orbits are contained 
in the closures of generic torus orbits. 
Such Demazure $P$-roots are defined by 
free generators of the Cox ring and their 
description is analogous to the toric case.
Secondly, there are ``horizontal'' Demazure $P$-roots
corresponding to root subgroups whose orbits 
are transversal to generic torus orbits. 
Dealing with this type heavily involves the 
relations among generators of the Cox ring.
Our first main result expresses the roots of $\Aut(X)^0$
and, moreover, the approach shows 
how to obtain the corresponding root subgroups,
see Theorem~\ref{thm:autroots} and 
Corollary~\ref{cor:autXgen} for the 
precise formulation:

\begin{introthm}
Let $X$ be a nontoric normal complete rational 
variety with an effective torus action 
$T \times X \to X$ of complexity one.
Then $\Aut(X)$ is a linear algebraic group 
having $T$ as a maximal torus and the 
roots of $\Aut(X)$ with respect to $T$ 
are precisely the $\ZZ^s$-parts of the Demazure
$P$-roots.
\end{introthm}

The basic idea of the proof is to relate the
group $\Aut(X)$ to the group of graded automorphisms 
of the Cox ring. 
This is done in Section~\ref{sec:autmds} 
more generally for arbitrary Mori dream spaces, 
i.e.~normal complete varieties with a finitely 
generated Cox ring $\mathcal{R}(X)$.
In this setting, the grading by the divisor class 
group $\Cl(X)$ defines an action of the characteristic 
quasitorus $H_X = \Spec \, \KK[\Cl(X)]$ on the total
coordinate space $\b{X} = \Spec \, \mathcal{R}(X)$
and $X$ is the quotient of an open subset 
$\rq{X} \subseteq \b{X}$ by the action of $H_X$.
The group of $\Cl(X)$-graded automorphisms of 
$\mathcal{R}(X)$ is isomorphic to the 
group $\Aut(\b{X},H_X)$ of $H_X$-equivariant 
automorphisms of $\b{X}$. 
Moreover, the group $\Bir_2(X)$ of birational 
automorphisms of $X$ defined on an open subset
of $X$ having complement of codimension at 
least two plays a role.
Theorem~\ref{thm:weakaut} brings all groups 
together:

\begin{introthm} 
Let $X$ be a (not necessarily rational) 
Mori dream space.
Then there is a commutative diagram
of morphisms of linear algebraic 
groups where the rows are exact 
sequences and the upwards inclusions 
are of finite index:
$$
\xymatrix{
1 
\ar[r]
&
H_X 
\ar[r]
&
{\Aut(\b{X},H_X)}
\ar[r]
&
{\Bir_2(X)} 
\ar[r]
&
1
\\ 
1 
\ar[r]
&
H_X 
\ar[r]
\ar@{=}[u]
&
{\Aut(\rq{X},H_X)}
\ar[r]
\ar@{^{(}->}[u]
&
{\Aut(X)} 
\ar[r]
\ar@{^{(}->}[u]
&
1
}
$$
\end{introthm}

This means in particular that the unit component 
of $\Aut(X)$ coincides with that of $\Bir_2(X)$ 
which in turn is determined by $\Aut(\b{X},H_X)$,
the group of graded automorphisms of the Cox ring.
Coming back to rational varieties $X$ with torus 
action of complexity one, the task then is a 
detailed study of the graded automorphism group 
of the rings $\mathcal{R}(X) = R(A,P)$.
This is done in a purely algebraic way.
The basic concepts are provided in 
Section~\ref{sec:factgrad}.
The key result is the description of the ``primitive 
homogeneous locally nilpotent derivations'' on 
$R(A,P)$ given in Theorem~\ref{thm:basiclnd}.
The proof of the first main theorem 
in Section~\ref{sec:demroots}
then relates the Demazure $P$-roots via these 
derivations to the roots of the automorphism 
group $\Aut(X)$.

In Section~\ref{sec:delpezzo} we apply our 
results to the study of almost homogeneous 
rational $\KK^*$-surfaces $X$ of 
Picard number one; where almost homogeneous
means that $\Aut(X)$ has an open orbit in $X$.
It turns out that these surfaces are always
(possibly singular) del Pezzo surfaces and,
up to isomorphism, there are countably many 
of them, see Corollary~\ref{thm:alhomdelp}.
Finally in the case that $X$ is log 
terminal with only one singularity, we give
classifications for fixed Gorenstein index.

In Section ~\ref{sec:semi}, we investigate 
the semisimple part 
$\Aut(X)^{\rm ss} \subseteq \Aut(X)$
of the automorphism group;
recall that the semisimple part of a 
linear algebraic group is a maximal 
connected semisimple subgroup.
In the case of a toric variety,
by Demazure's results, the semisimple part 
of the automorphism group has a root system 
composed of systems $A_i$. 
Here comes a summarizing version of our 
second main result which settles the 
complexity one case; see 
Theorem~\ref{thm:sesipart} for the
detailed description.

\begin{introthm}
Let $X$ be a nontoric normal complete rational 
variety with an effective torus action 
$T \times X \to X$ of complexity one.
The root system $\Phi$ of the semisimple part
splits as 
$\Phi = \Phi^{\rm vert} \oplus \Phi^{\rm hor}$
with
$$ 
\Phi^{\rm vert} \ = \ \bigoplus_{\Cl(X)} A_{m_D-1},
\qquad\qquad
\Phi^{\rm hor} \  \in \ \{\emptyset, A_1,A_2,A_3,A_1\oplus A_1,B_2\},
$$
where $m_D$ is the number of invariant prime
divisors in $X$ with infinite $T$-isotropy that 
represent a given class $D \in \Cl(X)$.
The number $m_D$ as well as the possibilities 
for $\Phi^{\rm hor}$ can be read off from the 
defining matrix $P$.
\end{introthm}

Examples and applications of this result
are discussed in Section~\ref{sec:appsem}.
The main results concern varieties of dimension three 
which are almost homogeneous 
under the action of a reductive group and
additionally admit an effective action of a 
two-dimensional torus. 
In Proposition~\ref{prop:dim3alhomred}, we
explicitly describe the Cox rings of these
varieties.
Moreover, in Proposition~\ref{prop:3dimautred},
we list all those having Picard number one 
and a reductive automorphism group; it turns 
out that any such variety is a Fano variety.

\tableofcontents

\section{The automorphism group of a Mori dream space}
\label{sec:autmds}

Let $X$ be a normal complete variety defined over 
an algebraically closed field $\KK$ of 
characteristic zero with finitely generated 
divisor class group $\Cl(X)$ and Cox 
sheaf~$\mathcal{R}$; we recall the definition 
below.
If $X$ is a Mori dream space, i.e.~the Cox ring 
$\mathcal{R}(X) = \Gamma(X,\mathcal{R})$ is 
finitely generated as a $\KK$-algebra, 
then we obtain the following picture
$$ 
\xymatrix{
{\Spec_X \mathcal{R}}
\ar@{}[r]|{\quad=}
& 
{\rq{X}}
\ar@{}[r]|\subseteq
\ar[d]_{\quot H_X}
&
{\b{X}}
\ar@{}[r]|{=\quad}
&
{\Spec \, \mathcal{R}(X)}
\\
& 
X
&
&
}
$$
where the {\em total coordinate space\/} $\b{X}$
comes with an action of the 
{\em characteristic quasitorus\/} 
$H_X := \Spec \, \KK[\Cl(X)]$,
the {\em characteristic space\/} $\rq{X}$, 
i.e. the relative spectrum of the Cox sheaf,
occurs as an open $H_X$-invariant 
subset of $\b{X}$ and the map 
$p_X \colon \rq{X} \to X$ 
is a good quotient for the action of~$H_X$.

We study automorphisms of $X$ in terms of 
automorphisms of $\b{X}$ and $\rq{X}$.
By an {\em $H_X$-equivariant automorphism\/} of $\b{X}$ 
we mean a pair $(\varphi,\t{\varphi})$, where 
$\varphi \colon \b{X} \to \b{X}$ is an 
automorphism of varieties
and $\t{\varphi} \colon H_X \to H_X$ is 
an automorphism of linear algebraic groups
satisfying
$$
\varphi(t \mal x) \ = \ \t{\varphi}(t) \mal \varphi(x)
\quad
\text{for all }
x \in \b{X}, \ t \in H_X.
$$
We denote the group of $H_X$-equivariant automorphisms
of $\b{X}$ by $\Aut(\b{X},H_X)$.
Analogously, one defines the group $\Aut(\rq{X},H_X)$
of $H_X$-equivariant automorphisms of $\rq{X}$.
A {\em weak automorphism\/} of $X$ is a birational 
map $\varphi \colon X \to X$  which defines an 
isomorphism of big open subsets, i.e., there are
open subsets $U_1,U_2 \subseteq X$ with complement 
$X \setminus U_i$ of codimension at least two in $X$ 
such that $\varphi_{\vert U_1} \colon U_1 \to U_2$ is a 
regular isomorphism.
We denote the group of weak automorphisms of $X$ 
by $\Bir_2(X)$.

\begin{theorem} 
\label{thm:weakaut}
Let $X$ be a Mori dream space.
Then there is a commutative diagram
of morphisms of linear algebraic 
groups where the rows are exact 
sequences and the upwards inclusions 
are of finite index:
$$
\xymatrix{
1 
\ar[r]
&
H_X 
\ar[r]
&
{\Aut(\b{X},H_X)}
\ar[r]
&
{\Bir_2(X)} 
\ar[r]
&
1
\\ 
1 
\ar[r]
&
H_X 
\ar[r]
\ar@{=}[u]
&
{\Aut(\rq{X},H_X)}
\ar[r]
\ar@{^{(}->}[u]
&
{\Aut(X)} 
\ar[r]
\ar@{^{(}->}[u]
&
1
}
$$
Moreover, there is a big open subset $U \subseteq X$ 
with $\Aut(U) = \Bir_2(X)$ 
and the groups $\Aut(\b{X},H_X)$, $\Bir_2(X)$,
$\Aut(\rq{X},H_X)$, $\Aut(X)$ act morphically on 
$\b{X}$, $U$, $\rq{X}$, $X$, respectively.
\end{theorem}

Our proof uses some ingredients from algebra
which we develop first.
Let $K$ be a finitely generated abelian group 
and consider a finitely generated 
integral $\KK$-algebra 
\begin{eqnarray*}
R 
& = & 
\bigoplus_{w \in K} R_w.
\end{eqnarray*} 
The {\em weight monoid\/} of $R$ is the submonoid 
$S \subseteq K$ consisting of the elements $w \in K$ 
with $R_w \ne 0$.
The {\em weight cone\/} of $R$ is the convex cone
$\omega \subseteq K_\QQ$ in the rational vector space 
$K_\QQ = K \otimes_\ZZ \QQ$ generated by the weight 
monoid $S \subseteq K$.
We say that the $K$-grading of $R$ is {\em pointed\/} if 
the weight cone $\omega \subseteq K_\QQ$ contains no line
and $R_0 = \KK$ holds.
By an {\em automorphism\/} of the $K$-graded algebra 
$R$ we mean a pair $(\psi,F)$,
where $\psi \colon R \to R$ is an isomorphism of 
$\KK$-algebras and $F \colon K \to K$ is an 
isomorphism such that $\psi(R_w) = R_{F(w)}$
holds for all $w \in K$.
We denote the group of such automorphisms 
of $R$ by $\Aut(R,K)$.

\goodbreak

\begin{proposition}
\label{prop:pointaut}
Let $K$ be a finitely generated abelian group 
and $R = \oplus_{w \in K} R_w$ a finitely generated 
integral $\KK$-algebra with $R^*=\KK^*$.
Suppose that the grading is pointed.
Then $\Aut(R,K)$ is a linear algebraic 
group over $\KK$ and $R$ is a rational 
$\Aut(R,K)$-module.
\end{proposition}

\begin{proof}
The idea is to represent the automorphism group 
$\Aut(R,K)$ as a closed subgroup of 
the linear automophism group of a suitable 
finite dimensional vector subspace 
$V^0 \subseteq R$.
In the subsequent construction of $V^0$,
we may assume that the weight cone $\omega$
generates $K_\QQ$ as a vector space.

Consider the subgroup $\Gamma \subseteq \Aut(K)$ 
of $\ZZ$-module automorphisms $K \to K$ such that 
the induced linear isomorphism $K_\QQ \to K_\QQ$ 
leaves the weight cone $\omega \subseteq K_\QQ$
invariant. 
By finite generation of $R$, the cone $\omega$
is polyhedral and thus $\Gamma$ is finite.
Let $f_1, \ldots, f_r \in R$ be homogeneous generators
and denote by $w_i := \deg(f_i) \in K$ their degrees.
Define a finite $\Gamma$-invariant subset 
and a vector subspace
$$ 
S^0 \ := \ \Gamma \cdot \{w_1, \ldots, w_r\} \subseteq \ K,
\qquad
V^0
\ := \ 
\bigoplus_{w \in S^0} R_w
\ \subseteq \
R.
$$

For every automorphism $(\psi,F)$ of the graded 
algebra $R$, we have $F(S^0) = S^0$ and thus 
$\psi(V^0) = V^0$. 
Moreover, $(\psi,F)$ is uniquely determined by its
restriction on $V^0$.
Consequently, we may regard the automorphism 
group $H := \Aut(R,K)$ as a subgroup of 
the general linear group $\GL(V^0)$.
Note that every $g \in H$ 
\begin{enumerate}
\item
permutes the components $R_w$ of 
the decomposition $V^0 = \bigoplus_{w \in S^0} R_w$,
\item
satisfies
$\sum_\nu a_\nu g(f_1)^{\nu_1} \cdots g(f_r)^{\nu_r} = 0$
for any relation
$\sum_\nu a_\nu f_1^{\nu_1} \cdots f_r^{\nu_r} =0$.
\end{enumerate}
Obviously, these are algebraic conditions.
Moreover, every $g \in \GL(V^0)$
satisfying the above conditions can be extended 
uniquely to an element of $\Aut(R,K)$
via 
$$
g \left( \sum_\nu a_\nu f_1^{\nu_1} \cdots f_r^{\nu_r} \right)
\ := \ 
\sum_\nu a_\nu g(f_1)^{\nu_1} \cdots g(f_r)^{\nu_r}.
$$
Thus, we saw that $H \subseteq \GL(V^0)$ is 
precisely the closed subgroup defined by the 
above conditions (i) and (ii).
In particular $H = \Aut(R,K)$ is linear algebraic.
Moreover, the symmetric algebra $SV^0$ is a rational 
$\GL(V^0)$-module,
hence $SV^0$ is a rational $H$-module for the algebraic 
subgroup $H$ of $\GL(V^0)$, and so is its factor module $R$. 
\end{proof}

\begin{corollary}
\label{cor:autalg}
Let $K$ be a finitely generated abelian group 
and $R = \oplus_{w \in K} R_w$ a finitely generated 
integral $\KK$-algebra with $R^*=\KK^*$.
Consider the corresponding action of 
$H := \Spec \, \KK[K]$ on $\b{X} := \Spec \, R$.
Then we have a canonical isomorphism
$$ 
\Aut(\b{X},H) \ \to \ \Aut(R,K),
\qquad
(\varphi,\t{\varphi})
\ \mapsto \ 
(\varphi^*,\t{\varphi}^*),
$$
where $\varphi^*$ is the pullback of regular functions 
and $\t{\varphi}^*$ the pullback of characters.
If the $K$-grading 
is pointed,
then $\Aut(\b{X},H)$ is a linear algebraic group 
acting morphically on $\b{X}$.
\end{corollary}

We will also need details of the construction 
of the Cox sheaf $\mathcal{R}$ on $X$ 
which we briefly recall now.
Denote by $c \colon \WDiv(X) \to \Cl(X)$
the map sending the Weil divisors to their classes,
let $\PDiv(X) = \ker(c)$ denote the group of 
principal divisors and choose  a character, 
i.e.~a group homomorphism 
$\chi \colon \PDiv(X) \to \KK(X)^*$ with
$$
\div(\chi(E))
\ = \ 
E,
\qquad
\text{for all } 
E \in \PDiv(X).
$$
This can be done by prescribing $\chi$ suitably
on a $\ZZ$-basis of $\PDiv(X)$.
Consider the associated sheaf of divisorial 
algebras 
$$ 
\mathcal{S}
\ := \ 
\bigoplus_{\WDiv(X)} \mathcal{S}_D,
\qquad\qquad
\mathcal{S}_D
\ := \ 
\mathcal{O}_X(D).
$$
Denote by $\mathcal{I}$ the sheaf of ideals 
of $\mathcal{S}$ locally generated by the 
sections $1 - \chi(E)$, where~$1$ is homogeneous
of degree zero, $E$ runs through $\PDiv(X)$ and 
$\chi(E)$ is homogeneous of degree $-E$.
The {\em Cox sheaf\/} associated to 
$K$ and $\chi$ is the quotient sheaf 
$\mathcal{R} := \mathcal{S}/\mathcal{I}$
together with the $\Cl(X)$-grading
$$ 
\mathcal{R}
\ = \
\bigoplus_{[D] \in \Cl(X)}  \mathcal{R}_{[D]},
\qquad\qquad
\mathcal{R}_{[D]} 
\ := \ 
\pi \left( \bigoplus_{D' \in c^{-1}([D])} \mathcal{S}_{D'} \right).
$$
where $\pi \colon \mathcal{S} \to \mathcal{R}$ 
denotes the projection.
The Cox sheaf $\mathcal{R}$ is a quasicoherent sheaf 
of $\Cl(X)$-graded $\mathcal{O}_X$-algebras.
The {\em Cox ring\/} is the ring $\mathcal{R}(X)$
of global sections  of the Cox sheaf.

\begin{proof}[Proof of Theorem~\ref{thm:weakaut}]
We set $G_{\b{X}} := \Aut(\b{X},H_X)$ for short.
According to Corollary~\ref{cor:autalg}, 
the group $G_{\b{X}}$ is linear algebraic 
and acts morphically on $\b{X}$.
Looking at the representations of $H_X$ and 
$G_{\b{X}}$ on 
$\Gamma(\b{X},\mathcal{O}) = \mathcal{R}(X)$ 
defined by the respective actions, we see that 
the canonical inclusion $H_X \to G_{\b{X}}$ 
is a morphism of linear algebraic groups.

Next we construct the subset $U \subseteq X$ from
the last part of the statement.
Consider the translates $g \mal \rq{X}$, where 
$g \in G_{\b{X}}$.
Each of them admits a good quotient with a complete 
quotient space:
$$
p_{X,g} \colon 
g \mal \rq{X} \ \to \ (g \mal \rq{X}) \quot H_X.
$$
By~\cite{BB1}, there are only finitely 
many open subsets of $\b{X}$ with such a good quotient.
In particular, the number of translates 
$g \mal \rq{X}$ is finite.

Let $W \subseteq X$ denote the maximal open 
subset such that the restricted quotient 
$\rq{W} \to W$, where $\rq{W} := p_X^{-1}(W)$,
is geometric, i.e.~has the $H_X$-orbits as its
fibers.
Then, for any $g \in  G_{\b{X}}$, the translate
$g \mal \rq{W} \subseteq g \mal \rq{X}$ 
is the (unique) maximal open subset which is 
saturated with respect to the quotient map $p_{X,g}$ 
and defines a geometric quotient. 
Consider
$$ 
\rq{U} 
\ := \ 
\bigcap_{g \in G_{\b{X}}} g \mal \rq{W}
\ \subseteq  \ 
\rq{X}.
$$
By the preceding considerations $\rq{U}$ 
is open, and by construction it is 
$G_{\b{X}}$-invariant and saturated 
with respect to $p_X$.
By~\cite[Prop.~6.1.6]{ArDeHaLa} the set $\rq{W}$ 
is big in $\b{X}$.
Consequently, also $\rq{U}$ is big in 
$\b{X}$.
Thus, the (open) set $U := p_X(\rq{U})$ is 
big in $X$.
By the universal property of the geometric 
quotient, there is a unique morphical
action of $G_{\b{X}}$ on $U$ making 
$p_X \colon \rq{U} \to U$ equivariant.
Thus, we have homomorphism of groups
$$ 
G_{\b{X}} 
\ \to \
\Aut(U)
\ \subseteq \ 
\Bir_2(X).
$$

We show that $\pi \colon G_{\b{X}} \to \Bir_2(X)$ 
is surjective. 
Consider a weak automorphism $\varphi \colon X \to X$. 
The pullback  defines an automorphism 
of the group of Weil divisors
$$
\varphi^* \colon \WDiv(X) \to \WDiv(X),
\qquad
D \ \mapsto \ \varphi^*D.
$$
As in the construction of the Cox sheaf, 
consider the sheaf of divisorial 
algebras $\mathcal{S} = \oplus \mathcal{S}_D$
associated to $\WDiv(X)$ and fix a character
$\chi \colon \PDiv(X) \to \KK(X)^*$ with
$\div(\chi(E))=E$ for any $E\in\PDiv(X)$. 
Then we obtain a homomorphism
$$
\alpha \colon \PDiv(X) \ \to \ \KK^*,
\qquad 
E \mapsto 
\frac{\varphi^*(\chi(E))}{\chi(\varphi^*(E))}.
$$
We extend this to a homomorphism 
$\alpha \colon \WDiv(X) \to \KK^*$ as follows.
Write $\Cl(X)$ as a direct sum of a free part
and cyclic groups $\Gamma_1,\ldots,\Gamma_s$ 
of order $n_i$.
Take $D_1, \ldots, D_r \in \WDiv(X)$ such that 
the classes of $D_1,\ldots, D_s$ are generators 
for $\Gamma_1,\ldots,\Gamma_s$ and the remaining 
ones define a basis of the free part. Set
$$ 
\alpha(D_i) \ := \ \sqrt[n_i]{\alpha(n_iD_i)}
\text{ for } 1 \le i \le s,
\qquad
\alpha(D_i) \ := \ 1 
\text{ for } s+1 \le i \le r.
$$ 
Then one directly checks that this extends 
$\alpha$ to a homomorphism $\WDiv(X) \to \KK^*$.
Using $\alpha(E)$ as a ``correction term'', 
we define an automorphism of the graded sheaf
$\mathcal{S}$ of divisorial algebras:
for any open set $V \subseteq X$ we set
$$
\varphi^* \colon 
\Gamma (V,\mathcal{S}_D)
\ \to \ 
\Gamma (\varphi^{-1}(V), \mathcal{S}_{\varphi^*(D)}), 
\qquad 
f  
\ \mapsto \ 
\alpha(D) f \circ \varphi.
$$
By construction $\varphi^*$ sends the ideal 
$\mathcal{I}$ arising from the character 
$\chi$ to itself.
Consequently, $\varphi^*$ descends to an 
automorphism $(\psi,F)$ 
of the (graded) Cox sheaf 
$\mathcal{R}$; note that
$F$ is the pullback of divisor classes
via $\varphi$.
The degree zero part of $\psi$ equals
the usual pullback of regular functions on 
$X$ via $\varphi$.
Thus, the element in $\Aut(\b{X},H_X)$
defined by 
$\Spec \, \psi \colon \rq{U} \to \rq{U}$ 
maps to $\varphi$.

Clearly, $H_X$ lies in the 
kernel of $\pi \colon G_{\b{X}} \to \Bir_2(X)$.
For the reverse inclusion, consider 
an element $g \in \ker(\pi)$. 
Then $g$ is a pair $(\varphi,\t{\varphi})$
and, by the construction of $\pi$, we have 
a commutative diagram 
$$ 
\xymatrix{
{\rq{U}}
\ar[r]^{\varphi}
\ar[d]_{p_X}
&
{\rq{U}}
\ar[d]^{p_X}
\\
U 
\ar[r]_{\id}
&
U
}
$$
In particular, $\varphi$ stabilizes
all $H_X$-invariant divisors.
It follows that the pullback $\varphi^*$
on $\Gamma(\rq{U},\mathcal{O})= \mathcal{R}(X)$
stabilizes the homogeneous components.
Thus, for any homogeneous $f$ of degree $w$, 
we have $\varphi^*(f) = \lambda(w)f$ with 
a homomorphism $\lambda \colon K \to \KK^*$.
Consequently $\varphi(x) = h \mal x$ holds
with an element $h \in H_X$.
The statements concerning the upper 
sequence are verified.

Now, consider the lower sequence. Since $\rq{X}$ 
is big in $\b{X}$, every automorphism of $\rq{X}$ 
extends to an automorphism of $\b{X}$.
We conclude that $\Aut(\rq{X},H_X)$ 
is the (closed) subgroup of  $G_{\b{X}}$ 
leaving the complement 
$\b{X} \setminus \rq{X}$ invariant.
As seen before, the collection of translates 
$G_{\b{X}} \cdot \rq{X}$ is finite and thus  
the subgroup $\Aut(\rq{X},H_X)$ of $G_{\b{X}}$
is of finite index.
Moreover, lifting $\varphi \in \Aut(X)$ as
before gives an element of $\Aut(\b{X},H_X)$ 
leaving $\rq{X}$ invariant.
Thus, $\Aut(\rq{X},H_X) \to \Aut(X)$ is surjective
with kernel $H_X$.
By the universal property of the qood quotient
$\rq{X} \to X$, the action of $\Aut(X)$ 
on $X$ is morphical.
\end{proof}

\begin{corollary}
\label{cor:autmds}
The automorphism group $\Aut(X)$ of a Mori 
dream space $X$ is linear algebraic and acts
morphically on $X$.
\end{corollary}

\begin{corollary}
If two Mori dream spaces are isomorphic in 
codimension one, then the unit components 
of their automorphism groups are isomorphic 
to each other.
\end{corollary}

Let $\CAut(\b{X},H_X)$ denote the centralizer 
of $H_X$ in the automorphism group $\Aut(\b{X})$.
Then $\CAut(\b{X},H_X)$ consists of all automorphisms
$\varphi \colon \b{X} \to \b{X}$ satisfying
$$
\varphi(t \mal x) \ = \ t \mal \varphi(x)
\quad
\text{for all }
x \in \b{X}, \ t \in H_X.
$$
In particular, we have 
$\CAut(\b{X},H_X) \subseteq \Aut(\b{X},H_X)$.
The group $\CAut(\b{X},H_X)$ 
may be used to detect the unit component 
$\Aut(X)^0$ of the 
automorphism group of $X$.

\begin{corollary}
Let $X$ be a Mori dream space.
Then there is an exact sequence
of linear algebraic groups
$$
\xymatrix{
1 
\ar[r]
&
H_X 
\ar[r]
&
{\CAut(\b{X},H_X)^0}
\ar[r]
&
{\Aut(X)^0} 
\ar[r]
&
1.
}
$$
\end{corollary}

\begin{proof}
According to~\cite[Cor.~2.3]{Sw}, 
the group $\CAut(\b{X},H_X)^0$ leaves $\rq{X}$ 
invariant.
Thus, we have 
$\CAut(\b{X},H_X)^0 \subseteq \Aut(\rq{X},H_X)$
and the sequence is well defined.
Moreover, for any $\varphi \in \Aut(X)^0$,
the pullback $\varphi^* \colon \Cl(X) \to \Cl(X)$ 
is the identity.
Consequently, $\varphi$ lifts to an element of 
$\CAut(\b{X},H_X)$.
Exactness of the sequence thus follows by dimension
reasons.
\end{proof}

\begin{corollary}
\label{cor:1pasglift}
Let $X$ be a Mori dream space.
Then, for any closed subgroup 
$F \subseteq \Aut(X)^0$, there is a closed 
subgroup $F' \subseteq \CAut(\b{X},H_X)^0$
such that the induced map $F' \to F$ is 
an epimorphism with finite kernel.
\end{corollary}

\begin{corollary}
Let $X$ be a Mori dream space such that
the group $\CAut(\b{X},H_X)$ is connected,
e.g. a toric variety.
Then there is an exact sequence of linear 
algebraic groups
$$
\xymatrix{
1 
\ar[r]
&
H_X 
\ar[r]
&
{\CAut(\b{X},H_X)}
\ar[r]
&
{\Aut(X)^0} 
\ar[r]
&
1.
}
$$
\end{corollary}

\begin{example}
Consider the nondegenerate quadric $X$ in the 
projective space $\PP_{n+1}$, where $n \ge 4$ 
is even.
Then the Cox ring of $X$ is the $\ZZ$-graded ring
$$ 
\mathcal{R}(X)
\ = \ 
\KK[T_0,\ldots,T_{n+1}] \, / \, \bangle{T_0^2 + \ldots + T_{n+1}^2} ,
\qquad
\deg(T_0) \ = \ \ldots \ = \ \deg(T_{n+1}) \ = \ 1.
$$
The characteristic quasitorus is $H_X = \KK^*$.
Moreover, for the equivariant automorphisms
and the centralizer of $H_X$ we obtain
$$
\Aut(\b{X},H_X) 
\ = \ 
\CAut(\b{X},H_X)
\ = \ 
\KK^*E_{n+2} \mal {\rm O}_{n+2}.
$$
Thus, $\CAut(\b{X},H_X)$ has two connected components.
Note that for $n=4$, the quadric $X$ comes 
with a torus action of complexity one. 
\end{example}

\section{Rings with a factorial grading of complexity one}
\label{sec:factgrad}

Here we recall 
the necessary constructions and results on factorially 
graded rings of complexity one and Cox rings 
of varieties with a torus action of complexity one
from~\cite{HaHe}.
The main result of this section is 
Proposition~\ref{prop:pan}
which describes the dimension of the homogeneous 
components in terms of the (common) degree of 
the relations.
As before, we work over an algebraically closed field
$\KK$ of characteristic zero.

Let $K$ be an abelian group and 
$R = \oplus_K R_w$ a $K$-graded algebra. 
The grading is called {\em effective\/} 
if the weight monoid $S$ of $R$ generates $K$ 
as a group.
Moreover, we say that the grading is of 
{\em complexity one}, if it is effective 
and $\dim(K_\QQ)$ equals $\dim(R)-1$.
By a {\em $K$-prime\/} element of $R$ we 
mean a homogeneous nonzero nonunit $f \in R$
such that $f \mid gh$ with homogeneous $g,h \in R$ 
implies $f \mid g$ or $f \mid h$.
We say that $R$ is {\em factorially $K$-graded\/} 
if every nonzero homogeneous nonunit of $R$ 
is a product of $K$-primes.

\begin{construction}
\label{constr:RAP}
See~\cite[Section~1]{HaHe}.
Fix $r \in \ZZ_{\ge 1}$, a sequence 
$n_0, \ldots, n_r \in \ZZ_{\ge 1}$, 
set $n := n_0 + \ldots + n_r$ and 
let $m \in \ZZ_{\ge 0}$.
The input data are 
\begin{itemize}
\item 
a matrix $A := [a_0, \ldots, a_r]$ 
with pairwise linearly independent 
column vectors $a_0, \ldots, a_r \in \KK^2$,
\item 
an integral $r \times (n+m)$ block matrix 
$P_0 = (L_0,0)$, where $L_0$ is a $r \times n$ 
matrix build from tuples 
$l_i := (l_{i1}, \ldots, l_{in_i}) \in \ZZ_{\ge 1}^{n_i}$ 
as follows
\begin{eqnarray*}
L_0
& = & 
\left[
\begin{array}{cccc}
-l_0 & l_1 &   \ldots & 0 
\\
\vdots & \vdots   & \ddots & \vdots
\\
-l_0 & 0 &\ldots  & l_{r} 
\end{array}
\right].
\end{eqnarray*}
\end{itemize}
Consider the polynomial ring 
$\KK[T_{ij},S_k]$ in the variables 
$T_{ij}$, where 
$0 \le i \le r$, $1 \le j \le n_i$ 
and $S_k$, where $1 \le k \le m$.
For every $0 \le i \le r$, define a monomial
\begin{eqnarray*}
T_i^{l_i} 
&  := &
T_{i1}^{l_{i1}} \cdots T_{in_i}^{l_{in_i}}.
\end{eqnarray*}
Denote by $\mathfrak{I}$ the set of 
all triples $I = (i_1,i_2,i_3)$ with 
$0 \le i_1 < i_2 < i_3 \le r$ 
and define for any $I \in \mathfrak{I}$ 
a trinomial 
\begin{eqnarray*}
g_I
& := & 
\det
\left[
\begin{array}{ccc}
T_{i_1}^{l_{i_1}} & T_{i_2}^{l_{i_2}} & T_{i_3}^{l_{i_3}}
\\
a_{i_1} & a_{i_2} & a_{i_3}
\end{array}
\right].
\end{eqnarray*}
Let $P^*_0$ denote the transpose of $P_0$.
We introduce a grading on $\KK[T_{ij},S_k]$ 
by the factor group $K_0 := \ZZ^{n+m}/\rm{im}(P^*_0)$.
Let $Q_0 \colon \ZZ^{n+m} \to K_0$ be the 
projection and set  
$$ 
\deg(T_{ij}) 
\ := \ 
w_{ij}
\ := \ 
Q_0(e_{ij}),
\qquad
\deg(T_{k}) 
\ := \ 
w_{k}
\ := \ 
Q_0(e_{k}),
$$
where $e_{ij} \in \ZZ^{n+m}$, for 
$0 \le i \le r$, $1 \le j \le n_i$,
and $e_{k} \in \ZZ^{n+m}$, for 
$1 \le k \le m$, are the canonical basis 
vectors.
Note that all the $g_{I}$ are $K_0$-homogeneous
of degree
$$
\mu 
\ := \ 
l_{01}w_{01} + \ldots + l_{0n_0}w_{0n_0}
\ = \ 
\ldots
\ = \ 
l_{r1}w_{r1} + \ldots + l_{rn_r}w_{rn_r}
\ \in \ 
K_0.
$$
In particular, the trinomials $g_I$ generate a $K_0$-homogeneous 
ideal and thus we obtain a $K_0$-graded factor algebra 
\begin{eqnarray*}
R(A,P_0)
& := &
\KK[T_{ij},S_k] 
\ / \
\bangle{g_{I}; \; I \in \mathfrak{I}}.
\end{eqnarray*}
\end{construction}

\begin{theorem}
See~\cite[Theorems~1.1 and~1.3]{HaHe}.
With the notation of Construction~\ref{constr:RAP},
the following statements hold.
\begin{enumerate}
\item
The $K_0$-grading of ring $R(A,P_0)$ is effective, pointed, 
factorial and of complexity one.
\item
The variables $T_{ij}$ and $S_k$ define a system of 
pairwise nonassociated $K_0$-prime generators of 
$R(A,P_0)$.
\item
Every finitely generated normal $\KK$-algebra with an
effective, pointed, factorial grading of complexity 
one is isomorphic to some $R(A,P_0)$.
\end{enumerate}
\end{theorem}

Note that in the case $r=1$, there are no relations 
and the theorem thus treats the effective, pointed
gradings of complexity one of the polynomial ring.

\begin{example}[The $E_6$-singular cubic I]
\label{ex:E6sing1}
Let $r = 2$, $n_0 = 2$, $n_1 = n_2 = 1$, $m=0$ 
and consider the data 
$$ 
A 
\ = \ 
\left[
\begin{array}{ccc} 
0 & -1 & 1
\\
1 & -1 & 0 
\end{array}
\right],
\qquad\qquad
P_0 
\ = \ 
L_0
\ = \ 
\left[
\begin{array}{rrrr} 
-1 & -3 & 3 & 0
\\
-1 & -3 & 0 & 2 
\end{array}
\right].
$$
Then we have exactly one triple in $\mathfrak{I}$, namely $I = (1,2,3)$, 
and, as a ring, $R(A,P_0)$ is given by 
\begin{eqnarray*}
R(A,P_0)
& = & 
\KK[T_{01},T_{02},T_{11},T_{21}] 
\ / \ 
\bangle{T_{01}T_{02}^3 + T_{11}^3 + T_{21}^2}.
\end{eqnarray*} 
The grading group 
$K_0 = \ZZ^4/{\rm im}(P_0^*)$ is isomorphic to $\ZZ^2$ 
and the grading can be given explicitly via
$$ 
\deg(T_{01}) 
\ = \ 
\left(
\begin{array}{r}
-3
\\
3
\end{array}
\right),
\qquad
\deg(T_{02}) 
\ = \ 
\left(
\begin{array}{r}
1
\\
1
\end{array}
\right),
$$
$$
\deg(T_{11}) 
\ = \ 
\left(
\begin{array}{r}
0
\\
2
\end{array}
\right),
\qquad
\deg(T_{21}) 
\ = \ 
\left(
\begin{array}{r}
0
\\
3
\end{array}
\right).
$$
\end{example}

Recall that for any integral ring 
$R = \oplus_K R_w$ graded by an 
abelian group $K$, one has the subfield 
of degree zero fractions inside the field of
fractions:
$$
Q(R)_0
\ = \ 
\left\{
\frac{f}{g}; 
\; 
f,g \in R \text{ homogeneous}, 
g \ne 0,
\deg(f) = \deg(g)
\right\}
\ \subseteq \ 
Q(R).
$$

\begin{proposition}
\label{prop:degzerofield}
Take any $i,j$ with $i \ne j$ and $0 \le i,j \le r$.
Then the field of degree zero fractions of the 
ring $R(A,P_0)$ is the rational function field
\begin{eqnarray*}
Q(R(A,P_0))_0
& = & 
\KK \left(\frac{T_j^{l_j}}{T_i^{l_i}}\right).
\end{eqnarray*}
\end{proposition}

\begin{proof}
It suffices to treat the case $m=0$.
Let $F = \prod T_{ij}$ be the product
of all variables. 
Then $\TT^n = \KK^n_F$ is the $n$-torus 
and $P_0$ defines an epimorphism
having the quasitorus $H_0 := \Spec \, \KK[K_0]$
as its kernel
$$ 
\pi \colon \TT^n \ \to \ \TT^r,
\qquad
(t_{ij}) \ \mapsto \ 
\left( 
\frac{t_1^{l_1}}{t_0^{l_0}},
\ldots,
\frac{t_r^{l_r}}{t_0^{l_0}}
\right).
$$
Set $\b{X} := \Spec \, R(A,P_0)$. Then
$\pi(\b{X}_F) = \b{X}_F / H_0$ is a curve defined by 
affine linear equations in the coordinates 
of $\TT^r$ and thus rational.
The assertion follows.
\end{proof}

The following observation shows that the 
common degree $\mu = \deg(g_I)$ of the 
relations generalizes the ``remarkable weight'' 
introduced by Panyushev~\cite{Pa} in the 
factorial case.
Recall that the weight monoid $S_0 \subseteq K_0$ 
consists of all $w \in K_0$ admitting 
a nonzero homogeneous element.

\begin{proposition} 
\label{prop:pan}
Consider the $K_0$-graded ring $R := R(A,P_0)$ 
and the degree $\mu = \deg(g_I)$ of the relations
as defined in~\ref{constr:RAP}.
For $w \in S_0$ let $s_w \in \ZZ_{\ge 0}$ be the 
unique number with 
$w - s_w\mu \in S_0$ and $w-(s_w+1) \mu \not\in S_0$.
Then we have\ 
$$
\dim(R_w) \ = \ s_w+1
\text{ for all } 
w \in S_0.
$$
The element $\mu \in K_0$ is uniquely determined by this 
property.
We have $\dim(R_{\mu})=2$ and any two 
nonproportional elements in $R_{\mu}$ are coprime.
Moreover, any $w \in S_0$ with $w - \mu \not\in S_0$ 
satisfies $\dim(R_w) = 1$.
\end{proposition}

\begin{proof}
According to Proposition~\ref{prop:degzerofield},
the field $Q(R)_0$ of degree zero fractions is 
the field of rational functions in $p_1/p_0$, 
where $p_0 := T_0^{l_0}$ and $p_1 := T_1^{l_1}$ 
are coprime and of degree $\mu$.
Moreover, by the structure of the relations $g_I$, 
we have $\dim(R_\mu) = 2$. 

Now, consider $w \in S_0$. If we have $\dim(R_w) = 1$,
then  $\dim(R_\mu) = 2$ implies $s_w = 0$ and the 
assertion follows in this case.
Suppose that we have $\dim(R_w) > 1$. 
Then we find two nonproportional elements 
$f_0,f_1 \in R_w$
and two coprime homogeneous polynomials $F_0$, $F_1$
of a common degree $s > 0$ such that
$$ 
\frac{f_1}{f_0}
\ = \ 
\frac{F_1(p_0,p_1)}{F_0(p_0,p_1)}.
$$
Observe that $F_1(p_0,p_1)$ must divide $f_1$.
This implies $w - s \mu \in S_0$.
Repeating the procedure with $w-s \mu $ and so on, we 
finally arrive at a weight $\t{w} = w - s_w \mu$ 
with $\dim(R_{\t{w}}) = 1$.
Moreover, by the procedure, any element of $R_w$ 
is of the form $hF(p_0,p_1)$ with $0 \ne h \in R_{\t{w}}$
and a homogeneous polynomial $F$ of degree $s_w$.
The assertion follows.
\end{proof}

\begin{corollary}
\label{cor:gencompdim}
Assume that we have $r \ge 2$ and 
that $l_{i1}+\ldots+l_{in_i} \ge 2$ holds 
for all $i$.
\begin{enumerate}
\item 
The $K_0$-homogeneous components 
$R(A,P)_{w_{ij}}$ and $R(A,P)_{w_{k}}$ 
of the generators $T_{ij}$ and $S_k$ are all of 
dimension one.
\item
Consider
$w  = w_{i_1j_1} + \ldots + w_{i_tj_t} \in K_0$,
where $1 \le t \le r$.
If $l_{i_kj_k} = 1$ holds for $1 \le k \le t-1$, 
then $R(A,P_0)_w$ is of dimension one.
\end{enumerate}
\end{corollary}

\begin{proof}
According to Proposition~\ref{prop:pan}, we have to 
show that the shifts of the weights 
$w_{ij},w_k$ and $w$ by $- \mu$ do not belong to the 
weight monoid $S_0$.
For $w_k$ this is clear. 
For $w_{ij}$,  the assumption gives
$$ 
w_{ij} - \mu 
\ = \ 
- (l_{ij} -1)w_{ij} - \sum_{b \ne j} l_{ib} w_{ib}
\ \not\in \
S_0.
$$
Let us consider the weight $w$ of~(ii).
Since $t \le r$ holds, there is an index
$0 \le i_0 \le r$ with $i_0 \ne i_k$ for
$k = 1, \ldots, t$. 
We have $w - \mu = Q_0(e)$ for
$$ 
e
\ := \ 
e_{i_1j_1} + \ldots + e_{i_tj_t}
-
(l_{i_01}e_{i_01} + \ldots + l_{i_0n_{i_0}}e_{i_0n_{i_0}})
\ \in \ 
\ZZ^{n+m}.
$$
By the assumptions, we find
$1 \le a_i \le n_i$, where $0 \le i \le r$,
such that $a_{i_k} \ne j_k$ holds for 
$1 \le k \le t-1$
and $a_{i_t} \ne j_t$ or $l_{i_ta_{i_t}} \ge 2$.
Then the linear form
$$
l_{0a_0}^{-1}e_{0a_0}^* + \ldots + l_{ra_r}^{-1}e_{ra_r}^*
$$
vanishes along the kernel 
of $Q_0 \colon \QQ^{n+m} \to (K_0)_\QQ$ and  thus
induces a linear form on~$(K_0)_\QQ$ which 
separates $w - \mu = Q_0(e)$ 
from the weight cone.
\end{proof}

We turn to Cox rings of varieties with a 
complexity one torus action.
They are obtained by suitably downgrading 
the rings $R(A,P_0)$ as 
follows.

\begin{construction}
\label{constr:RAPdown}
Fix $r \in \ZZ_{\ge 1}$, a sequence 
$n_0, \ldots, n_r \in \ZZ_{\ge 1}$, set 
$n := n_0 + \ldots + n_r$, and fix  
integers $m \in \ZZ_{\ge 0}$ and $0 < s < n+m-r$.
The input data are 
\begin{itemize}
\item 
a matrix $A := [a_0, \ldots, a_r]$ 
with pairwise linearly independent 
column vectors $a_0, \ldots, a_r \in \KK^2$,
\item 
an {\em integral block matrix\/} $P$ of size 
$(r + s) \times (n + m)$ the columns 
of which are pairwise different primitive
vectors generating $\QQ^{r+s}$ as a cone: 
\begin{eqnarray*}
P
& = & 
\left( 
\begin{array}{cc}
L_0 & 0 
\\
d & d'  
\end{array}
\right),
\end{eqnarray*}
where $d$ is an $(s \times n)$-matrix, $d'$ an $(s \times m)$-matrix 
and $L_0$ an $(r \times n)$-matrix build from
tuples $l_i := (l_{i1}, \ldots, l_{in_i}) \in \ZZ_{\ge 1}^{n_i}$ 
as in~\ref{constr:RAP}.
\end{itemize}
Let $P^*$ denote the transpose of $P$,
consider the factor group 
$K := \ZZ^{n+m}/\rm{im}(P^*)$
and the projection $Q \colon \ZZ^{n+m} \to K$.
We define a $K$-grading on 
$\KK[T_{ij},S_k]$ by setting
$$ 
\deg(T_{ij}) 
 \ := \ 
Q(e_{ij}),
\qquad
\deg(S_{k}) 
 \ := \ 
Q(e_{k}).
$$
The trinomials $g_I$ of~\ref{constr:RAP} 
are $K$-homogeneous, all of the same degree.
In particular, we obtain a $K$-graded 
factor ring  
\begin{eqnarray*}
R(A,P)
& := &
\KK[T_{ij},S_k; \; 0 \le i \le r, \, 1 \le j \le n_i, 1 \le k \le m] 
\ / \
\bangle{g_I; \; I \in \mathfrak{I}}.
\end{eqnarray*}
\end{construction}

\begin{theorem}
See~\cite[Theorem~1.4]{HaHe}.
With the notation of Construction~\ref{constr:RAPdown},
the following statements hold.
\begin{enumerate}
\item
The $K$-grading of the ring $R(A,P)$ is factorial,
pointed and almost free, i.e.~$K$ 
is generated by any $n+m-1$ of the $\deg(T_{ij}), \deg(S_k)$.
\item
The variables $T_{ij}$ and $S_k$ define a system of 
pairwise nonassociated $K$-prime generators of 
$R(A,P)$.
\end{enumerate}
\end{theorem}

\begin{remark}
\label{rem:downgrade}
As rings $R(A,P_0)$ and $R(A,P)$ coincide but 
the $K_0$-grading is finer than the $K$-grading.
The downgrading map $K_0 \to K$ fits into the
following commutative diagram built from exact 
sequences
$$ 
\xymatrix@R=15pt{
&
&
&
0
\ar[d]
&
\\
&
0
\ar[d]
&
&
{\ZZ^s}
\ar[d]
&
\\
0
\ar[r]
&
{\ZZ^r}
\ar[r]^{P_0^*}
\ar[d]
&
{\ZZ^{n+m}}
\ar[r]^{Q_0}
\ar@{=}[d]
&
K_0
\ar[r]
\ar[d]
&
0
\\
0
\ar[r]
&
{\ZZ^{r+s}}
\ar[r]_{P^*}
\ar[d]
&
{\ZZ^{n+m}}
\ar[r]_{Q}
&
K
\ar[r]
\ar[d]
&
0
\\
&
{\ZZ^s}
\ar[d]
&
&
0
&
\\
&
0
&
&
&
}
$$
The snake lemma~\cite[Sec.~III.9]{La} allows us to identify the 
direct factor $\ZZ^s$ of $\ZZ^{r+s}$ with the kernel of 
the downgrading map $K_0 \to K$.
Note that for the quasitori $T$, $H_0$ and $H$ associated 
to abelian groups $\ZZ^s$, $K_0$ and $K$ we have $T = H_0/H$.
\end{remark}

\begin{construction}
\label{constrRAP2TVar}
Consider a ring $R(A,P)$ with its $K$-grading 
and the finer $K_0$-grading. 
Then the quasitori $H = \Spec \, \KK[K]$ and 
$H_0 := \Spec \, \KK[K_0]$ act on 
$\b{X} := \Spec \, R(A,P)$.
Let $\rq{X} \subseteq \b{X}$ be a 
big $H_0$-invariant open subset 
with a good quotient 
$$
p \colon \rq{X} \ \to \ X = \rq{X} \quot H
$$
such that $X$ is complete 
and for some open set $U \subseteq X$, 
the inverse image $p^{-1}(U) \subseteq \b{X}$ 
is big and $H$ acts freely on $U$.
Then $X$ is a Mori dream space of dimension 
$s+1$ with divisor class group 
$\Cl(X) \cong K$ and Cox ring 
$\mathcal{R}(X) \cong R(A,P)$.
Moreover, $X$ comes with an induced 
effective action of the $s$-dimensional 
torus $T := H_0/H$.
\end{construction}

\begin{remark}
Let $X$ be a $T$-variety arising from data
$A$ and $P$ via Construction~\ref{constrRAP2TVar}.
Then every $T_{ij} \in R(A,P)$ defines an
invariant prime divisor $D_{ij} = \b{T \mal x_{ij}}$ 
in $X$ such that the isotropy group $T_{x_{ij}}$
is cyclic of order $l_{ij}$ 
and the class of $d_{ij} \in \ZZ^s$ 
in $(\ZZ/l_{ij}\ZZ)^s$ 
is the weight vector of the tangent presentation 
of $T_{x_{ij}}$ at $x_{x_{ij}}$.
Moreover, each $S_k$ defines an invariant 
prime divisor $E_k \subseteq X$ such that 
the one-parameter subgroup $\KK^* \to T$
corresponding to $d_k' \subseteq \ZZ^s$ 
acts trivially on $E_k$.
\end{remark}

\begin{theorem}
\label{thm:allCompl1var}
Let $X$ be an $n$-dimensional complete normal 
rational variety with an effective action 
of an $(n-1)$-dimensional torus $S$.
Then $X$ is equivariantly isomorphic to 
a $T$-variety arising from data $(A,P)$ 
as in~\ref{constrRAP2TVar}.
\end{theorem}

\begin{proof}
We may assume that $X$ is not a toric variety.
According to~\cite[Theorem~1.5]{HaHe}, the 
$\Cl(X)$-graded Cox ring of $X$ is isomorphic 
to a $K$-graded ring $R(A,P)$.
Thus, in the notation of~\ref{constrRAP2TVar},
there is a big $H$-invariant open subset 
$\rq{X}$ of $\b{X}$ with $X \cong \rq{X} \quot H$.
Applying~\cite[Cor.~2.3]{Sw} to a subtorus 
$T_0 \subseteq H_0$ projecting onto $T = H_0/H$, 
we see that $\rq{X}$ is even invariant under $H_0$.
Thus, $T$ acts on $X$. Since the $T$-action is 
conjugate in $\Aut(X)$ to the given $S$-action 
on $X$, the assertion follows.
\end{proof}

\begin{example}[The $E_6$-singular cubic II]
\label{ex:E6sing2}
Let $r = 2$, $n_0 = 2$, $n_1 = n_2 = 1$, $m=0$,
$s=1$ and consider the data 
$$ 
A 
\ = \ 
\left[
\begin{array}{ccc} 
0 & -1 & 1
\\
1 & -1 & 0 
\end{array}
\right],
\qquad\qquad
P \ = \ 
\left[
\begin{array}{rrrr} 
-1 & -3 & 3 & 0
\\
-1 & -3 & 0 & 2 
\\
-1 & -2 & 1 & 1 
\end{array}
\right].
$$
Then, as remarked before, we have exactly one triple $I = (1,2,3)$ 
and, as a ring, $R(A,P)$ is given by 
\begin{eqnarray*}
R(A,P)
& = & 
\KK[T_{01},T_{02},T_{11},T_{21}] 
\ / \ 
\bangle{T_{01}T_{02}^3 + T_{11}^3 + T_{21}^2}.
\end{eqnarray*} 
The grading group 
$K = \ZZ^4/{\rm im}(P^*)$ is isomorphic to $\ZZ$ 
and the grading can be given explicitly via
$$ 
\deg(T_{01}) 
\ = \ 
3,
\quad
\deg(T_{02}) 
\ = \ 
1,
\quad
\deg(T_{11}) 
\ = \
2,
\quad
\deg(T_{21}) 
\ = \ 
3.
$$
As shown for example in~\cite{HaTschi},
see also \cite[Example~3.7]{Ha3L}, 
the ring $R(A,P)$ is the Cox ring of the 
$E_6$-singular cubic surface in the projective
space given by
$$  
X 
\ = \ 
V(z_1z_2^2 + z_2z_0^2 + z_3^3)
\ \subseteq \ 
\PP_3.
$$
\end{example}

\section{Primitive locally nilpotent derivations}

Here, we investigate the homogeneous locally nilpotent 
derivations of the $K_0$-graded algebra $R(A,P_0)$.
The description of the ``primitive'' ones given 
in Theorem~\ref{thm:basiclnd} is the central algebraic 
tool for our study of automorphism groups.
As before, $\KK$ is an algebraically closed field 
of characteristic zero.

Let us briefly recall the necessary background.
We consider derivations on an integral $\KK$-algebra 
$R$, that means $\KK$-linear maps $\delta \colon R \to R$ 
satisfying the Leibniz rule 
\begin{eqnarray*}
\delta(fg) 
& = & 
\delta(f)g \ + \ f\delta(g).
\end{eqnarray*}
Any such $\delta \colon R \to R$ extends uniquely 
to a derivation $\delta \colon Q(R) \to Q(R)$ of the 
quotient field.
Recall that a derivation $\delta \colon R \to R$ is 
said to be {\em locally nilpotent\/} if for every 
$f \in R$ there is an $n \in \NN$ with 
$\delta^n(f) = 0$.
Now suppose that $R$ is graded by 
a finitely generated abelian group:
\begin{eqnarray*}
R & = & \bigoplus_{w \in K} R_w.
\end{eqnarray*}
A derivation $\delta \colon R \to R$ is 
called {\em homogeneous\/} if for every $w \in K$ 
there is a $w' \in K$ with 
$\delta(R_w) \subseteq R_{w'}$.
Any homogeneous derivation $\delta \colon R \to R$
has a {\em degree\/} $\deg(\delta) \in K$ satisfying
$\delta(R_w) \subseteq R_{w+\deg(\delta)}$ for 
all $w \in K$.

\begin{definition}
Let $K$ be a finitely generated abelian group,
$R = \oplus_K R_w$ a $K$-graded $\KK$-algebra
and $Q(R)_0 \subseteq Q(R)$ the subfield of 
all fractions $f/g$ of homogeneous elements 
$f,g \in R$ with $\deg(f) = \deg(g)$.
\begin{enumerate}
\item
We call a homogeneous derivation 
$\delta \colon R \to R$ {\em primitive\/} 
if $\deg(\delta)$ does not lie in 
the weight cone $\omega \subseteq K_\QQ$ 
of $R$.
\item
We say that a homogeneous derivation 
$\delta \colon R \to R$ is of 
{\em vertical type\/} if $\delta(Q(R)_0) = 0$
holds and of {\em horizontal type\/}
otherwise.
\end{enumerate}
\end{definition}

\begin{construction}
\label{constr:phlnd}
Notation as in Construction~\ref{constr:RAP}.
We define derivations of the $K_0$-graded 
algebra $R(A,P_0)$ 
constructed there. The input data are
\begin{itemize}
\item
a sequence $C = (c_0, \ldots, c_r)$ with 
$1 \le c_i \le n_i$,
\item
a vector $\beta \in \KK^{r+1}$ 
lying in the row space of the
matrix $[a_0, \ldots, a_r]$. 
\end{itemize}
Note that for $0 \ne \beta$ as above either all entries 
differ from zero or there is a unique $i_0$ with 
$\beta_{i_0} = 0$. 
According to these cases, we put further conditions 
and define:
\begin{enumerate}
\item
if all entries $\beta_0, \ldots, \beta_r$ differ
from zero 
and there is at most one $i_1$ with $l_{i_1c_{i_1}} > 1$, 
then we set
\begin{eqnarray*}
\delta_{C,\beta}(T_{ij})
& := & 
\begin{cases}
\beta_i \prod_{k \ne i} \frac{\partial T_k^{l_k}}{\partial T_{k c_k}},
& \quad 
j = c_i,
\\
0, 
& \quad
j \ne c_i,
\end{cases}
\\
\delta_{C,\beta}(S_k)
& := & 
0 \quad \text{for } k = 1, \ldots, m,
\end{eqnarray*}
\item
if $\beta_{i_0} = 0$ is the unique zero entry of $\beta$
and there is at most one $i_1$ with $i_1 \ne i_0$ and 
$l_{i_1c_{i_1}} > 1$, 
then we set
\begin{eqnarray*}
\delta_{C,\beta}(T_{ij})
& := & 
\begin{cases}
\beta_i \prod_{k \ne i,i_0} \frac{\partial T_k^{l_k}}{\partial T_{k c_k}},
& \quad
j = c_i,
\\
0, 
& \quad
j \ne c_i,
\end{cases}
\\
\delta_{C,\beta}(S_k)
& := & 
0 \quad \text{for } k = 1, \ldots, m.
\end{eqnarray*}
\end{enumerate}
These assignments define $K_0$-homogeneous
primitive locally nilpotent derivations 
$\delta_{C,\beta} \colon R(A,P_0) \to R(A,P_0)$ 
of degree
\begin{eqnarray*}
\deg(\delta_{C,\beta}) 
& = & 
\begin{cases}
r\mu - \sum_k \deg(T_{kc_k}),
& 
\quad 
\text{in case (i)},
\\
(r-1)\mu - \sum_{k\ne i_0} \deg(T_{kc_k}),
& 
\quad
\text{in case (ii)}.
\end{cases}
\end{eqnarray*}
\end{construction}

\begin{proof}
The assignments~(i) and~(ii) on the variables
define a priori derivations of the polynomial ring
$\KK[T_{ij},S_k]$.
Recall from~\ref{constr:RAP} that $R(A,P_0)$ 
is the quotient of $\KK[T_{ij},S_k]$ by the ideal 
generated by 
all
\begin{eqnarray*}
g_I
& = & 
\det
\left[
\begin{array}{ccc}
T_{i_1}^{l_{i_1}} & T_{i_2}^{l_{i_2}} & T_{i_3}^{l_{i_3}}
\\
a_{i_1} & a_{i_2} & a_{i_3}
\end{array}
\right],
\end{eqnarray*}
where $I = (i_1,i_2,i_3)$.
Since the vector $\beta$ lies in the
row space of $[a_0, \ldots, a_r]$, 
we see that $\delta_{C,\beta}$ sends 
every trinomial $g_I$ to zero and thus 
descends to a well defined derivation
of $R(A,P_0)$.

We check that $\delta_{C,\beta}$ is homogeneous.
Obviously, every $\delta_{C,\beta}(T_{ij})$ is a 
$K_0$-homogeneous element of $\KK[T_{ij}]$.
Moreover, with the degree $\mu$ of the 
relations $g_I$, we have 
\begin{eqnarray*}
\deg(\delta_{C,\beta}(T_{ij})) - \deg(T_{ij}) 
& = & 
\begin{cases}
r\mu - \sum_k \deg(T_{kc_k}),
& 
\quad 
\text{in case (i)},
\\
(r-1)\mu - \sum_{k\ne i_0} \deg(T_{kc_k}),
& 
\quad
\text{in case (ii)}.
\end{cases}
\end{eqnarray*}
In particular, the left hand side does not depend 
on $(i,j)$. 
We conclude that $\delta_{C,\beta}$ is homogeneous
of degree $\deg(\delta_{C,\beta}(T_{ij})) - \deg(T_{ij})$. 

For primitivity, we have to show that the degree 
of $\delta_{C,\beta}$ does not lie in the weight 
cone of $R(A,P_0)$.
We exemplarily treat case~(i), where we may assume 
that $i_1 = 0$ holds.
As seen before, the degree of $\delta_{C,\beta}$ is 
represented by the vector
$$ 
v_{C,\beta}
\ := \ 
- e_{0c_0} 
+
\sum_{j \ne c_1} l_{1j} e_{1j}
+ 
\ldots 
+
\sum_{j \ne c_r} l_{rj} e_{rj}
\ \in \
\ZZ^{n+m}.
$$
Thus, we look for a linear form on $\QQ^{n+m}$ separating 
this vector from the orthant $\cone(e_{ij},e_k)$ and 
vanishing along the kernel of $\QQ^{n+m} \to (K_0)_\QQ$, 
i.e.~the linear subspace spanned by the columns of $P_0^*$.
For example, we may take
$$ 
l_{0c_0}^{-1}e_{0c_0}^* 
+ 
l_{1c_1}^{-1}e_{1c_1}^* 
+ \ldots + 
l_{rc_r}^{-1}e_{rc_r}^*.
$$

Finally, we show that $\delta_{C,\beta}$ is locally 
nilpotent.
If $l_{ic_i} = 1$ holds for all $i$, then 
$\delta_{C,\beta}(T_{ij})$ is a product 
of variables from the kernel of $\delta_{C,\beta}$
and thus $\delta_{C,\beta}^2$ annihilates 
all generators.
If $l_{0c_0} > 1$ holds, then we have
$$
\delta_{C,\beta}^2(T_{0c_0}) \ = \ 0,
\qquad
\qquad
\delta_{C,\beta}(T_{ic_i}) \ = \ T_{0c_0}^{l_{0c_0}-1}h_i,
$$
where $i \ge 1$ and $h_i$ lies in the kernel of 
$\delta_{C,\beta}$.
Putting all together, we obtain that $\delta_{C,\beta}^{l_{0c_0}+1}$
annihilates all generators $T_{ij}$ and thus 
$\delta_{C,\beta}$ is locally nilpotent.
\end{proof}

\begin{theorem}
\label{thm:basiclnd}
Let 
$\delta \colon R(A,P_0) 
\to R(A,P_0)$ 
be a nontrivial 
primitive $K_0$-homogeneous locally nilpotent 
derivation.
\begin{enumerate}
\item
If $\delta$ is of vertical type, then 
$\delta(T_{ij}) = 0$ holds for all $i,j$
and there is a $k_0$ such that $\delta(S_{k_0})$ 
does not depend on $S_{k_0}$ and $\delta(S_k) = 0$ 
holds for all $k \ne k_0$.
\item
If $\delta$ is of horizontal type, then
we have $\delta = h\delta_{C,\beta}$, where 
$\delta_{C,\beta}$ is as in \ref{constr:phlnd}
and $h$ is $K_0$-homogeneous with 
$h \in \ker(\delta_{C,\beta})$.
\end{enumerate}
\end{theorem}

In the proof of this theorem we will make frequently 
use of the following facts; the statements of the first 
Lemma occur in Freudenburg's book, 
see~\cite[Principles~1, 5 and~7, Corollary~1.20]{Fr}.

\begin{lemma} 
\label{lem:LND}
Let $R$ be an integral $\KK$-algebra, $\delta \colon R \to R$ 
a locally nilpotent derivation and let $f,g \in R$.
\begin{enumerate}
\item
If $fg \in \ker(\delta)$ holds,
then $f,g \in \ker(\delta)$ holds.
\item
If $\delta(f)=fg$ holds, then $\delta(f)=0$ holds.
\item
The derivation $f\delta$ is locally nilpotent 
if and only if $f\in\ker(\delta)$ holds.
\item
If $g \mid \delta(f)$ and $f \mid \delta(g)$,
then $\delta(f) = 0$ or $\delta(g)=0$.
\end{enumerate}
\end{lemma}

\begin{lemma}
\label{lem:lift}
Let 
$\delta \colon R(A,P_0)  \to R(A,P_0)$ 
be a primitive $K_0$-homogeneous derivation
induced by a $K_0$-homogeneous derivation
$\widehat{\delta} \colon \KK[T_{ij},S_k] \to \KK[T_{ij},S_k]$.
Then $\widehat{\delta}(g_I) = 0$ holds for all relations $g_I$.
\end{lemma}

\begin{proof}
Clearly, we have $\widehat{\delta}(\mathfrak{a}) \subseteq \mathfrak{a}$ 
for the ideal $\mathfrak{a} \subseteq \KK[T_{ij},S_k]$ 
generated by the $g_I$.
Recall that all $g_I$ are of the same degree $\mu$.
By primitivity, $\deg(\delta) = \deg(\widehat{\delta})$
is not in the weight cone.
Thus, $\KK[T_{ij}]_{\mu+\deg(\widehat{\delta})} \cap \mathfrak{a}
= \{ 0 \}$ holds.
This implies $\widehat{\delta}(g_I) = 0$.
\end{proof}

\begin{proof}[Proof of Theorem~\ref{thm:basiclnd}]
Suppose that $\delta$ is of vertical type.
Then $\delta(T_i^{l_i}/T_s^{l_s}) = 0$ holds
for any two $0 \le i < s \le r$.
By the Leibniz rule, this implies
\begin{eqnarray*}
\delta(T_i^{l_i})T_s^{l_s} 
& = & 
T_i^{l_i}\delta(T_s^{l_s}).
\end{eqnarray*}
We conclude that $T_i^{l_i}$ divides $\delta(T_i^{l_i})$
and $T_s^{l_s}$ divides $\delta(T_s^{l_s})$.
By Lemma~\ref{lem:LND}~(ii), this implies 
$\delta(T_i^{l_i}) = \delta(T_s^{l_s}) = 0$.
Using Lemma~\ref{lem:LND}~(i), we obtain 
$\delta(T_{ij}) =  0$ for all variables $T_{ij}$.
Since $\delta$ is nontrivial,
we should have $\delta(S_{k_0})\ne 0$
at least for one~$k_0$. 
Consider the basis $e_k = \deg(S_k)$
of $\ZZ^m$, where $k=1,\ldots, m$,
and write
$$
\deg(\delta)
\ = \ 
w' + \sum_{k=1}^m b_ke_k , 
\quad \text{where} \quad w'\in K_0
\quad \text{and} \quad b_k \in \ZZ. 
$$
Then 
$\deg(\delta(S_{k_0}))=w'+\sum_{k \ne k_0} b_ke_k + (b_{k_0}+1)e_{k_0}$.
By Lemma~\ref{lem:LND}, the variable 
$S_{k_0}$ does not divide $\delta(S_{k_0})$.
This and the condition $\delta(S_{k_0})\ne 0$ imply
$b_{k_0}=-1$ and $b_k \ge 0$ for $k \ne k_0$.
This proves that 
$\delta(S_k)=0$ for all $k \ne k_0$ and
$\delta(S_{k_0})$ is $K_0$-homogeneous and 
does not depend on $S_{k_0}$.

Now suppose that $\delta$ is of horizontal type.
Then there exists a variable $T_{ij}$ with $\delta(T_{ij})\ne 0$. 
Write
\begin{eqnarray*} 
\deg(\delta(T_{ij}))
& = & 
\deg(T_{ij}) 
\ + \ 
w' 
\ + \ 
\sum_{k=1}^m b_ke_k.
\end{eqnarray*}
Then all coefficients $b_k$ are nonnegative 
and consequently we obtain $\delta(S_k)=0$ for 
$k = 1, \ldots, m$. 

We show that for any 
$T_i^{l_i}$ 
there is at most one variable $T_{ij}$ with 
$\delta(T_{ij}) \ne 0$.
Assume that we find two different $j,k$ 
with $\delta(T_{ij}) \ne 0$ and 
$\delta(T_{ik}) \ne 0$.
Note that we have 
$$
\frac{\partial T_i^{l_i}}{\partial T_{ij}} \delta(T_{ij}),
\
\frac{\partial T_i^{l_i}}{\partial T_{ik}} \delta(T_{ik})
\quad \in \quad 
R(A,P_0)_{\mu + \deg(\delta)}.
$$
By Proposition~\ref{prop:pan}, the component
of degree $\mu + \deg(\delta)$ is of dimension 
one.
Thus, the above two terms differ by a nonzero 
scalar and we see that $T_{ik}^{l_{ik}}$
divides the second term.
Consequently, $T_{ik}$ must divide $\delta(T_{ik})$
which contradicts~\ref{lem:LND}~(ii).

A second step is to see that for any two variables
$T_{ij}$ and $T_{ks}$ with $\delta(T_{ij}) \ne 0$ 
and $\delta(T_{ks}) \ne 0$ we must have $l_{ij}=1$ 
or $l_{ks} = 1$. 
Otherwise, we see as before that 
$\partial T_i^{l_i}/\partial T_{ij} \, \delta(T_{ij})$
and 
$\partial T_k^{l_k}/\partial T_{ks} \, \delta(T_{ks})$
differ by a nonzero scalar.
Thus, we conclude
$\delta(T_{ij}) =  fT_{ks}$ and 
$\delta(T_{ks})=  hT_{ij}$, 
a contradiction to~\ref{lem:LND}~(iv).

Finally, we prove the assertion.
As already seen, for every $0 \le k \le r$
there is at most one $c_k$ with 
$\delta(T_{kc_k}) \ne 0$.
Let $\mathfrak{K} \subseteq \{0,\ldots,r\}$ denote 
the set of all $k$ admitting such a $c_k$.
From Proposition~\ref{prop:pan} we infer 
$R(A,P)_{\mu+\deg(\delta)} = \KK f$ with 
some nonzero element $f$.
We claim that
$$ 
f 
\ = \ 
h 
\prod_{k \in \mathfrak{K}} \frac{\partial T_k^{l_k}}{\partial T_{kc_k}},
\qquad\qquad
\delta(T_{ic_i})
\ = \ 
\beta_i h 
\prod_{k \in \mathfrak{K} \setminus \{i\}} 
\frac{\partial T_k^{l_k}}{\partial T_{kc_k}},
$$
hold with a homogeneous element $h \in R(A,P)$ and 
scalars $\beta_0, \ldots, \beta_r \in \KK$.
Indeed, similar to the previous arguments, the first 
equation follows from fact that all 
$\partial T_k^{l_k}/\partial T_{kc_k} \, \delta(T_{kc_k})$
are nonzero elements of the same degree as $f$ 
and hence each $\partial T_k^{l_k}/\partial T_{kc_k}$ must 
divide $f$. The second equation is clear then.

The vector $\beta := (\beta_0, \ldots, \beta_r)$ 
lies in the row space of the matrix $A$.
To see this, consider the lift of $\delta$ to 
$\KK[T_{ij},S_k]$ defined by the second equation and 
apply Lemma~\ref{lem:lift}.
Now let $C = (c_0, \ldots, c_r)$ be any sequence completing 
the $c_k$, where $k \in \mathfrak{K}$.
Then we have $\delta = h\delta_{C,\beta}$.
The fact that $h$ belongs to the kernel of 
$\delta_{C,\beta}$ follows from Lemma~\ref{lem:LND}.
\end{proof}

\begin{example}[The $E_6$-singular cubic III]
\label{ex:E6sing3}
Situation as in~\ref{ex:E6sing1}.
The locally nilpotent primitive homogeneous derivations 
of $R(A,P_0)$ of the form $\delta_{C,\beta}$ 
are the following
\begin{enumerate}
\item
$C=(1,1,1)$ and $\beta = (\beta_0, 0, -\beta_0)$.
Here we have $\deg(\delta_{C,\beta})=(3,0)$ and 
$$
\qquad
\delta_{C,\beta}(T_{01})=2\beta_0T_{21}, \quad
\delta_{C,\beta}(T_{21})=-\beta_0T_{02}^3, \quad
\delta_{C,\beta}(T_{02})=\delta_{C,\beta}(T_{11})=0.
$$
\item
$C=(1,1,1)$ and $\beta=(\beta_0,-\beta_0,0)$.
Here we have $\deg(\delta_{C,\beta})=(3,1)$ and
$$
\qquad
\delta_{C,\beta}(T_{01})=3\beta_0T_{11}^2, \quad
\delta_{C,\beta}(T_{11})=-\beta_0T_{02}^3, \quad
\delta_{C,\beta}(T_{02})=\delta_{C,\beta}(T_{21})=0.
$$
\end{enumerate}
The general  locally nilpotent primitive homogeneous 
derivation $\delta$ of $R(A,P_0)$ has the form 
$h\delta_{C,\beta}$ with
$h\in\ker(\delta_{C,\beta})$, and
$$
\deg(\delta)
\ = \ 
\deg(h) + \deg(\delta_{C,\beta}) 
\ \notin \ 
\omega.
$$
In the above case~(i), the only possibilities for $\deg(h)$ are
$\deg(h) = (k,k)$ or $\deg(h) = (k,k) + (0,2)$
and thus we have 
$$
\delta \ =  \ T_{02}^k\delta_{C,\beta} \quad \text{or} \quad
\delta \ = \ T_{02}^kT_{11}\delta_{C,\beta}.
$$
In the above case~(ii), the only possibility for $\deg(h)$ is
$\deg(h) = (k,k)$ and thus we obtain
$$
\delta \ = \ T_{02}^k\delta_{C,\beta}.
$$
\end{example}

\section{Demazure roots}
\label{sec:demroots}

Here we present and prove the main result,
Theorem~\ref{thm:autroots}.
It describes the root system of the automorphism 
group of a rational complete normal variety~$X$ 
coming with an effective torus action 
$T \times X \to X$ of complexity one
in terms of the defining matrix $P$ of 
the Cox ring $\mathcal{R}(X) = R(A,P)$,
see Construction~\ref{constrRAP2TVar}
and Theorem~\ref{thm:allCompl1var}.

\begin{definition}
Let $A, P$ be as in 
Construction~\ref{constr:RAPdown}.
We say that $R(A,P)$ is {\em minimally 
presented\/} if $r \ge 2$ holds and 
for every $0 \le i \le r$ we have 
$l_{i1} + \ldots + l_{in_i} \ge 2$
\end{definition}

The assumption that  $R(A,P)$ is 
minimally presented
means that the resulting variety is
nontoric and there occur no linear 
monomials in the defining relations $g_I$; 
the latter can always be achieved by 
omitting redundant generators.

\begin{definition}
\label{def:Pdemroot}
Let $P$ be a matrix as in
Construction~\ref{constr:RAPdown}.
Denote by $v_{ij}, v_k \in N = \ZZ^{r+s}$
the columns of $P$ and
by $M$ the dual lattice of $N$.
\begin{enumerate}
\item
A {\em vertical Demazure $P$-root\/} is
a tuple $(u,k_0)$ with a linear form $u \in M$
and an index $1 \le k_0 \le m$ satisfying
\begin{eqnarray*}
\bangle{u,v_{ij}}
& \ge  &
0
\qquad \text{ for all } i,j,
\\
\bangle{u,v_{k}}
& \ge  &
0
\qquad \text{ for all } k \ne k_0,
\\
\bangle{u,v_{k_0}}
& =  &
-1.
\end{eqnarray*}
\item
A {\em horizontal Demazure $P$-root\/} 
is a tuple $(u,i_0,i_1,C)$, where
$u \in M$ is a linear form,
$i_0 \ne i_1$ are indices
with $0 \le i_0, i_1 \le r$,
and $C = (c_0,\ldots,c_r)$ is a sequence
with $1 \le c_i \le n_i$ such that
\begin{eqnarray*}
l_{ic_i}
& = &
1 \qquad \text{ for all } i \ne i_0,i_1,
\\
\bangle{u,v_{ic_i}}
& = &
\begin{cases}
0,        &\qquad i \ne i_0,i_1, 
\\
-1,        &\qquad i=i_1, 
\end{cases}
\\
\bangle{u,v_{ij}}
& \ge &
\begin{cases}
l_{ij}, &\qquad i \ne i_0,i_1, \quad j \ne c_i,
\\
0,      &\qquad i = i_0,i_1, \quad j \ne c_i,
\\
0,       &\qquad i=i_0,\hphantom{i_1,} \quad j = c_i,
\end{cases}
\\
\bangle{u,v_k}
& \ge &
0 \qquad \text{ for all } k.
\end{eqnarray*}
\item
The {\em $\ZZ^s$-part\/} of a Demazure $P$-root
$\kappa = (u,k_0)$ 
or $\kappa = (u,i_0,i_1,C)$ is the 
tuple $\alpha_\kappa$ of the last $s$ coordinates 
of the linear form $u \in M = \ZZ^{r+s}$.
We call $\alpha_\kappa$ also a {\em $P$-root}. 
\end{enumerate}
\end{definition}

Note that in the minimally presented case, 
the $P$-roots are by their defining conditions 
always nonzero.

\begin{example}[The $E_6$-singular cubic IV]
\label{ex:e6demProots}
As earlier, let $r = 2$, $n_0 = 2$, $n_1 = n_2 = 1$, 
$m=0$, $s=1$ and consider the data 
$$ 
A 
\ = \ 
\left[
\begin{array}{ccc} 
0 & -1 & 1
\\
1 & -1 & 0 
\end{array}
\right],
\qquad\qquad
P \ = \ 
\left[
\begin{array}{rrrr} 
-1 & -3 & 3 & 0
\\
-1 & -3 & 0 & 2 
\\
-1 & -2 & 1 & 1 
\end{array}
\right].
$$ 
There are no vertical Demazure $P$-roots because of $m=0$.
There is a horizontal Demazure $P$-root 
$\kappa = (u,i_0,i_1,C)$ given by
$$
u \ = \ (-1,-2,3), 
\qquad 
i_0 \ = \ 1, 
\quad 
i_1 \ = \ 2, 
\qquad
C \ = \ (1,1,1).
$$
A direct computation shows that this is the only one.
The $\ZZ^s$-part of $\kappa$ is the third coordinate 
of the linear form $u$, i.e.~ it is $u_3 = 3 \in \ZZ = \ZZ^s$.
\end{example}

Note that the Demazure $P$-roots are certain 
Demazure roots~\cite[Section~3.1]{Dem} 
of the fan with the rays through 
the columns of $P$ as its maximal cones.
In particular, there are only finitely many Demazure 
$P$-roots. 
For computing them explicitly, the following 
presentation is helpful.

\begin{remark}
\label{rem:DemPolytope}
The Demazure $P$-roots are 
the lattice points of certain polytopes 
in~$M_\QQ$.
For an explicit description, we encode 
the defining conditions as 
a lattice vector $\zeta \in \ZZ^{n+m}$ and 
an affine subspace $\eta \subseteq M_{\QQ}$:
\begin{enumerate}
\item 
For any index $1 \le k_0 \le m$ define a lattice 
vector $\zeta = (\zeta_{ij},\zeta_k) \in \ZZ^{n+m}$ 
and an affine subspace $\eta \subseteq M_{\QQ}$ 
by 
$$
\zeta_{ij} := 0
\text{ for all } i,j,
\quad
\zeta_k := 0 \text{ for all } k \ne k_0,
\quad
\zeta_{k_0} := -1,
$$
$$
\eta 
\ := \ 
\{u' \in M_\QQ; \bangle{u',v_{k_0}} = -1\}
\ \subseteq \ 
M_\QQ.
$$
Then the vertical Demazure $P$-roots 
$\kappa = (u,k_0)$ are given by the lattice 
points $u$ of the polytope
$$ 
B(k_0)
\ := \ 
\{
u' \in \eta; \ P^*u' \ge \zeta
\}
\ \subseteq \ 
M_\QQ.
$$
\item
Given $i_0 \ne i_1$ with $0 \le i_0,i_1 \le r$ and 
$C = (c_0,\ldots,c_r)$ with $1 \le c_i \le n_i$ such that
$l_{ic_i} = 1$ holds for all $i \ne i_0,i_1$, set
$$ 
\qquad\qquad
\zeta_{ij} 
\ := \ 
\begin{cases}
l_{ij}, & i \ne i_0,i_1, \ j \ne c_i,
\\
-1,     & i = i_1,\hphantom{i_0,} \ j = c_{i_1},
\\
0       & \text{else}, 
\end{cases}
\qquad\qquad
\zeta_k
\ = \ 
0
\text{ for } 1 \le l \le m.
$$  
$$
\qquad\qquad
\eta 
\ := \ 
\{
u' \in M_\QQ;  
\bangle{u',v_{ic_i}} = 0 \text{ for } i \ne i_0,i_1, 
\ 
\bangle{u',v_{i_1c_{i_1}}} = -1,
\}.
$$
Then the horizontal Demazure $P$-roots 
$\kappa = (u,i_0,i_1,C)$ 
are given by 
the lattice points $u$ of the polytope
$$ 
B(i_0,i_1,C)
\ := \ 
\{
u' \in \eta; \ P^*u' \ge \zeta
\}
\ \subseteq \ 
M_\QQ.
$$
\end{enumerate}
\end{remark}

In order to state and prove the main result, 
let us briefly recall the necessary
concepts from the theory of linear 
algebraic groups $G$. 
One considers the adjoint representation 
of the torus $T$ on the Lie algebra $\text{Lie}(G)$,
i.e.~the tangent representation at $e_G$ of 
the $T$-action on $G$ given by conjugation 
$(t,g) \mapsto tgt^{-1}$.
There is a unique $T$-invariant splitting 
$\Lie(G) = \Lie(T) \oplus \mathfrak{n}$,
where $\mathfrak{n}$ is spanned by nilpotent 
vectors, and one has a bijection
$$ 
\PASG_T(G)
\ \to \
\{T \text{-eigenvectors of } \mathfrak{n}\},
\qquad\qquad
\lambda 
\ \mapsto \ 
\dot \lambda (0).
$$
Here $\PASG_T(G)$ denotes the set of one 
parameter additive subgroups 
$\lambda \colon \GG_a \to G$ 
normalized by $T$ and $\dot \lambda$
denotes the differential.
A root of $G$ with respect to $T$ is 
an eigenvalue of the 
$T$-representation on  $\mathfrak{n}$, 
that means a character 
$\chi \in \Chi(T)$ with 
$t \cdot v = \chi(t) v$ for 
some $T$-eigenvector $0 \ne v \in  \mathfrak{n}$.

\begin{theorem}
\label{thm:autroots}
Let $A,P$ be as in Construction~\ref{constr:RAPdown}
such that $R(A,P)$ is minimally presented and 
let $X$ be a (nontoric) variety with a complexity one 
torus action $T \times X \to X$ arising from
$A,P$ according to Construction~\ref{constrRAP2TVar}.
\begin{enumerate}
\item
The automorphism group $\Aut(X)$ is a linear 
algebraic group with maximal torus $T$.
\item
Under the canonical identification 
$\Chi(T) = \ZZ^s$, the roots of 
$\Aut(X)$ with respect to $T$ 
are precisely the $P$-roots.
\end{enumerate}
\end{theorem}

The rest of the section is devoted to the proof.
We will have to deal with the $K_0$- and $K$-degrees
of functions and derivations.
It might be helpful to recall the relations between 
the gradings from Remark~\ref{rem:downgrade}.
The following simple facts will be frequently used.

\begin{lemma}
\label{lemmaKd0}
In the setting of Constructions~\ref{constr:RAP} 
and~\ref{constr:RAPdown}, 
consider the polynomial ring $\KK[T_{ij},S_k]$ with 
the $K_0$-grading and the coarser $K$-grading.
\begin{enumerate}
\item
For a monomial $h = \prod T_{ij}^{e_{ij}} \prod S_k^{e_k}$ 
with exponent vector $e = (e_{ij},e_k)$, 
the $K_0$- and $K$-degrees are given as 
$$ 
\deg_{K_0}(h) \ = \ Q_0(e),
\qquad\qquad
\deg_{K}(h) \ = \ Q(e).
$$
\item
A monomial $h \in \KK[T_{ij}^{\pm 1},S_k^{\pm{1}}]$ 
is of $K$-degree 
zero if and only if there is an $u \in M$ with 
$$ 
\qquad
\qquad
h 
\ = \ 
h^u
\ := \ 
\prod T_{ij}^{P^*(u)_{ij}} \prod S_k^{P^*(u)_k}
\ = \ 
\prod T_{ij}^{\bangle{u,v_{ij}}} 
\prod S_k^{\bangle{u,v_k}}.
$$
\item
Let $\delta$ be a derivation on $\KK[T_{ij},S_k]$
sending the generators $T_{ij},S_k$ to monomials.
Then $\delta$ is $K$-homogeneous of $K$-degree zero
if and only if 
$$
\qquad\qquad
\deg_K(T_{ij}^{-1}\delta(T_{ij}))
\ = \ 
\deg_K(S_k^{-1}\delta(S_k))
\ = \ 
0
\text{ holds for all } i,j,k.
$$
If $0 \ne \delta$ is $K_0$-homogeneous, then 
$\deg_K(\delta) = 0$ holds if and only if
one of the $T_{ij}^{-1}\delta(T_{ij})$ and 
$S_k^{-1}\delta(S_k)$ is nontrivial of 
$K$-degree zero.
\end{enumerate}
\end{lemma}

As a first step towards the roots of the 
automorphism group $\Aut(X)$, we now 
associate 
$K_0$-homogeneous locally nilpotent 
derivations of $R(A,P)$ to the 
Demazure $P$-roots.

\begin{construction}
\label{constr:DEMLND}
Let $A$ and $P$ be as in
Construction~\ref{constr:RAPdown}.
For $u \in M$ and the lattice vector 
$\zeta \in \ZZ^{n+m}$  
of Remark~\ref{rem:DemPolytope}
consider the monomials 
$$ 
h^u 
\ = \ 
\prod_{i,j} T_{ij}^{\bangle{u,v_{ij}}} 
\prod_{k} S_k^{\bangle{u,v_k}},
\qquad\qquad
h^\zeta 
\ := \ 
\prod_{i,j} T_{ij}^{\zeta_{ij}} 
\prod_{k} S_k^{\zeta_k}.
$$
We associate to any Demazure 
$P$-root $\kappa$ a locally nilpotent 
derivation $\delta_\kappa$ of $R(A,P)$.
If $\kappa = (u,k_0)$ is vertical, 
then we define a $\delta_\kappa$ of vertical 
type by
$$
\delta_\kappa(T_{ij}) \ := \ 0  \text{ for all } i,j,
\qquad
\delta_\kappa(S_k)
\ := \  
\begin{cases}
S_{k_0}h^u, & k = k_0,
\\
0,          & k \ne k_0.
\end{cases}
$$
If $\kappa = (u,i_0,i_1,C)$ is horizontal, 
then there is a unique vector $\beta$ in the 
row space of~$A$ with $\beta_{i_0}=0$, $\beta_{i_1}=1$ 
and we define a $\delta_\kappa$ of horizonal type by 
\begin{eqnarray*}
\delta_{\kappa} & := & \frac{h^u}{h^\zeta} \delta_{C,\beta}.
\end{eqnarray*}
In all cases, the derivation $\delta_\kappa$ is 
$K_0$-homogeneous; its $K_0$-degree is the $\ZZ^s$-part 
of~$\kappa$ and the $K$-degree is zero:
$$ 
\deg_{K_0}(\delta_\kappa) 
\ = \ 
Q_0(P^*(u)), 
\qquad\qquad
\deg_{K}(\delta_\kappa) 
\ = \ 
0.
$$
\end{construction}

\begin{proof}
In the vertical case $\delta_{\kappa}(S_{k_0})$ does not depend
on $S_{k_0}$ and 
in the horizontal case the factors before $\delta_{C,\beta}$
in the definitions of $\delta_{\kappa}$ are contained 
in $\ker(\delta_{C,\beta})$.
Thus, the derivations $\delta_{\kappa}$ are locally nilpotent.
Clearly, the $\delta_\kappa$ are $K_0$-homogeneous.
By Lemma~\ref{lemmaKd0},
the monomial $h^u$ is of $K_0$-degree $Q_0(P^*(u))$.
In the vertical case, this implies directly that 
$\delta_\kappa$ is of $K_0$-degree $Q_0(P^*(u))$.
In the horizontal case, we use  Lemma~\ref{lemmaKd0}
and the degree computation
of Construction~\ref{constr:phlnd} to see that $h^\zeta$ 
and $\delta_{C,\beta}$ have the same $K_0$-degree.
Thus $\delta_\kappa$ is of $K_0$-degree $Q_0(P^*(u))$.
Since $P^*(u) \in \ker(Q)$ holds, we obtain that 
all $\delta_{\kappa}$ are of $K$-degree $Q(P^*(u)) = 0$.
\end{proof}

\begin{proposition}
\label{prop:DEM2deriv}
Consider a  minimally presented algebra $R(A,P)$ 
with its fine $K_0$-grading 
and the coarser $K$-grading and
let $\delta$ be a $K_0$-homogeneous locally nilpotent 
derivation of $K$-degree zero on~$R(A,P)$.
\begin{enumerate}
\item
If $\delta$ is of vertical type, then there is 
an index $1 \le k_0 \le m$ such that $\delta$ 
is a linear combination of derivations 
$\delta_{\kappa_t}$ with Demazure $P$-roots 
$\kappa_t = (u_t,k_0)$.
\item
If $\delta$ is of horizontal type, then
are indices $0 \le i_0,i_1 \le r$ and 
a sequence $C = (c_0,\ldots, c_r)$ such 
that $\delta$ is a linear combination  
of derivations $\delta_{\kappa_t}$ with Demazure 
$P$-roots $\kappa_t = (u_t,i_0,i_1,C)$.
\end{enumerate}
\end{proposition}

\begin{lemma} 
\label{lemmaprim}
Let $\delta$ be a nontrivial $K_0$-homogeneous 
locally nilpotent derivation on a minimally presented 
algebra $R(A,P)$ and let $r \ge 2$.
If $\delta$ is of $K$-degree zero, then $\delta$ is 
primitive with respect to the $K_0$-grading.
\end{lemma}

\begin{proof}
We have to show that the $K_0$-degree $w$ of 
$\delta$ does not lie in the weight cone of the 
$K_0$-grading.
First observe that $w \ne 0$ holds:
otherwise Corollary~\ref{cor:gencompdim} 
yields that $\delta$ annihilates all generators
$T_{ij}$ and $S_k$, a contradiction to 
$\delta \ne 0$.
Now assume that $w$ lies in the weight cone 
of the $K_0$-grading. Then, for some $d > 0$,
we find a nonzero $f \in R(A,P)_{dw}$.
The $K$-degree of $f$ equals zero and thus
$f$ is constant, a contradiction.
\end{proof}

\begin{proof}[Proof of Proposition~\ref{prop:DEM2deriv}]
First assume that $\delta$ is vertical.
Lemma~\ref{lemmaprim} tells us that $\delta$ is primitive 
with respect to the $K_0$-grading.
According to Theorem~\ref{thm:basiclnd},
there is an index $1 \le k_0 \le m$
and an element $h \in R(A,P)$ represented 
by a polynomial only depending on variables 
from $\ker(\delta)$ such that we have 
$$
\delta(T_{ij}) \ = \ 0 \text{ for all } i,j,
\qquad
\delta(S_k) \ = \ 0 \text{ for all } k \ne k_0,
\qquad
\delta(S_{k_0}) \ = \ h.
$$
Clearly, $hS_{k_0}^{-1}$ is $K_0$-homogeneous 
of $K$-degree zero. 
Lemma~\ref{lemmaKd0} shows that the monomials 
of~$hS_{k_0}^{-1}$ are of the form $h^uS_{k_0}^{-1}$ 
with $u \in M$.
The facts that the monomials~$h^uS_{k_0}$ 
do not depend on~$S_{k_0}$ and have nonnegative 
exponents yield the inequalities of a vertical 
Demazure $P$-root for each $(u,k_0)$.
Consequently, $\delta$ is a linear combination 
of deriviations arising from vertical Demazure 
$P$-roots.

We turn to the case that $\delta$ is horizontal.
Again by Lemma~\ref{lemmaprim}, our $\delta$ 
is primitive with respect to the $K_0$-grading
and by Theorem~\ref{thm:basiclnd} it has the 
form $h \delta_{C,\beta}$ for some $K_0$-homogeneous
$h \in \ker(\delta_{C,\beta})$.
By construction, $\delta_{C,\beta}$ is induced 
by a homogeneous derivation of $\KK[T_{ij},S_k]$ 
having the same $K_0$- and $K$-degrees; 
we denote this lifted derivation again by 
$\delta_{C,\beta}$.
Similarly, $h$ is represented by 
a polynomial in $\KK[T_{ij},S_k]$
which we again denote by $h$.

We show that any monomial of $h$ depends only 
on variables from $\ker(\delta_{C,\beta})$.
Indeed, suppose that there occurs a monomial
$T_{ij}h'$ with $\delta_{C,\beta}(T_{ij})\ne 0$
in $h$. 
Then, using the fact that $\delta$ is of $K$-degree
zero, we obtain 
$$
\deg(T_{ij})
\ = \ 
\deg(\delta(T_{ij}))
\ = \ 
\deg(T_{ij}) + \deg(h')+\deg(\delta_{C,\beta}(T_{ij})).
$$
This implies $\deg(h')+\deg(\delta_{C,\beta}(T_{ij}))=0$;
a contradiction to the fact that the weight cone 
of the $K$-grading contains no lines.
This proves the claim.
Thus, we may assume that the polynomial~$h$ 
is a monomial.

The next step is to see that it is sufficient to 
take derivations $\delta_{C,\beta}$ with 
a vector $\beta$ in the row space having 
one zero coordinate.
Consider a general $\beta$, that means one 
with only nonvanishing coordinates.
By construction, the row space of $A$
contains unique vectors $\beta^0$ and $\beta^1$ 
with $\beta^0_0 = \beta^1_1 = 0$ 
and $\beta=\beta^0+\beta^1$. 
With these vectors, we have 
\begin{eqnarray*}
h\delta_{C,\beta} 
& = &
h\frac{\partial T_0^{l_0}}{\partial T_{0 c_0}}\delta_{C,\beta^0} 
\ + \
h\frac{\partial T_1^{l_1}}{\partial T_{1 c_1}}\delta_{C,\beta^1}.
\end{eqnarray*}
By Construction~\ref{constr:phlnd}, the $K_0$-degrees 
and thus the $K$-degrees of the left hand side and 
of the summands coincide.
Moreover, $h$ is a monomial in generators from 
$\ker(\delta_{C,\beta})$ and any such generator is 
annihilated by $\delta_{C,\beta^0}$ and by 
$\delta_{C,\beta^1}$ too.

Let $e = (e_{ij},e_k)$ denote the exponent 
vector of the monomial $h$.
According to Lemma~\ref{lemmaKd0}, 
the condition that the ($K_0$-homogeneous) 
derivation $\delta$ has $K$-degree zero is 
equivalent to the fact that the monomial
\begin{eqnarray*}
T_{i_1c_{i_1}}^{-1}h\delta_{C,\beta}(T_{i_1c_{i_1}})
& = &
T_{i_1c_{i_1}}^{-1}
T_{i_0c_{i_0}}^{e_{i_0c_{i_0}}}
\
\prod_{\genfrac{}{}{0pt}{}{i}{j\ne c_i}} T_{ij}^{e_{ij}} 
\
\prod_k S_k^{e_k}
\
\beta_{i_1}\prod_{i\ne i_0,i_1} \frac{\partial T_i^{l_i}}{\partial T_{ic_i}}
\end{eqnarray*}
has the form $h^u$ for some linear form $u \in M$. 
Taking into account that the exponents $e_{ij}$
and $e_k$ are nonnegative, we see that these 
conditions are equivalent to equalities and 
inequalities in the definition of a horizontal 
Demazure $P$-root.
\end{proof}

We recall the correspondence 
between locally nilpotent derivations
and one parameter additive subgroups.
Consider any integral affine 
$\KK$-algebra $R$,
where $\KK$ is an algebraically closed 
field of characteristic zero.
Every locally nilpotent derivation
$\delta \colon R \to R$ gives rise to
a rational representation 
$\varrho_\delta \colon \GG_a \to \Aut(R)$ 
of the additive group $\GG_a$ of the 
field $\KK$ via
$$
\varrho_\delta(t)(f)
\ := \ 
\exp(t\delta)(f)
\ := \
\sum_{d = 0}^{\infty} \frac{t^d}{d!} \delta^d(f).
$$
This sets up a bijection between the 
locally nilpotent derivations of $R$ 
and the rational representations 
of $\GG_a$ by automorphisms of $R$.
The representation associated to a 
locally nilpotent derivation 
$\delta \colon R \to R$ gives rise 
to a one parameter additive subgroup
(1-PASG) of the automorphism group
of $\b{X} := \Spec \, R$:
$$ 
\lambda_\delta \colon \GG_a \ \to \ \Aut(\b{X}),
\qquad
t \ \mapsto \ \Spec(\varrho_\delta(t)).
$$
Now suppose that $R$ is graded by some finitely 
generated abelian group $K_0$ and consider the 
associated action of $H_0 := \Spec \, \KK[K_0]$ 
on $\b{X} = \Spec \, R$. 
We relate homogeneity of locally 
nilpotent derivation $\delta$ to properties 
of the associated subgroup 
$\b{U}_\delta := \lambda_\delta(\GG_a)$ of 
$\Aut(\b{X})$.

\begin{lemma}
\label{lem:homogvsnormalize}
In the above setting, let $\delta$ be a locally 
nilpotent derivation on $R$.
The following statements are equivalent.
\begin{enumerate}
\item
The derivation $\delta$  is $K_0$-homogeneous.
\item
One has
$h \b{U}_\delta h^{-1} = \b{U}_\delta$ 
for all $h \in H_0$.
\end{enumerate}
Moreover, if one of these two statements holds, 
then the degree $w := \deg(\delta) \in K_0$ is uniquely
determined by the property
$$
h \varrho_\delta(t) h^{-1} \ = \ \varrho_\delta(\chi^w(h)t)
\text{ for all } h \in H_0.
$$
\end{lemma}

\begin{proof}[Proof of Theorem~\ref{thm:autroots}]
Assertion~(i) is clear by Corollary~\ref{cor:autmds}
and the fact that $X$ is nontoric.
We prove~(ii).
Consider $R(A,P)$ with its
fine $K_0$-grading and the coarser $K$-grading.
The quasitori $H_0 := \Spec \, \KK[K_0]$ 
and $H := \Spec \, \KK[K]$ act effectively on 
$\b{X} = \Spec \, R(A,P)$.
We view $H_0$ and $H$ as subgroups of 
$\Aut(\b{X})$.
For any locally nilpotent deriviation 
$\delta$ on $R(A,P)$ and 
$\b{U}_\delta = \lambda_\delta(\GG_a)$,
Lemma~\ref{lem:homogvsnormalize} 
gives
\begin{eqnarray*}
\delta \text{ is $K_0$-homogeneous}
& \iff & 
h \b{U}_\delta h^{-1} = \b{U}_\delta \text{ for all  } h \in H_0,
\\
\delta \text{ is $K$-homogeneous of degree } 0
& \iff & 
h u h^{-1} =  u \text{ for all  } h \in H, u \in \b{U}_\delta.
\end{eqnarray*} 
Recall that $X$ arises as $X = \rq{X} \quot H$ 
for an open 
$H_0$-invariant set $\rq{X} \subseteq \b{X}$.
Moreover, the action of $T = H_0/H$ on $X$ is the 
induced one, i.e. it makes the quotient map 
$p \colon \rq{X} \to X$ equivariant.
Set for short
$$
\b{G} \ := \ \CAut(\b{X},H)^0,
\qquad\qquad
G \ := \ \Aut(X)^0.
$$

Denote by $\PASG_{H_0}(\b{G})$ and $\PASG_{T}(G)$
the one parameter additive subgroups normalized by 
$H_0$ and $T$ respectively.
Moreover, let $\LND(R(A,P))_0$ denote the 
$K$-homogeneous locally nilpotent derivations 
of $K$-degree zero and $\LND_{K_0}(R(A,P))_0$
the subset of $K_0$-homogeneous ones.
Then we arrive at a commutative diagram
$$ 
\xymatrix{
{\LND_{K_0}(R(A,P))_0}
\ar@{<->}[d]_{\cong}
\ar@{}[rr]|\subseteq
& &
{\LND(R(A,P))_0}
\ar@{<->}[d]^{\cong}
\\
{\PASG_{H_0}(\b{G})}
\ar@{}[rr]|\subseteq
\ar[d]_{p_*}
& &
{\PASG(\b{G})}
\ar[d]^{p_*}
\\
{\PASG_{T}(G)}
\ar@{}[rr]|\subseteq
& & 
{\PASG(G)}
}
$$

Construction~\ref{constr:DEMLND} associates 
an element 
$\delta_\kappa \in \LND_{K_0}(R(A,P))_0$ to 
any Demazure $P$-root $\kappa$.
Going downwards the left hand side of the 
above diagram, the latter  turns into an 
element $\lambda_\kappa \in \PASG_{T}(G)$.
Differentiation gives the $T$-eigenvector 
$\dot \lambda_\kappa(0)  \in \Lie(G)$ having
as its associated root the unique 
character $\chi$ of $T$ satisfying 
$$
t \lambda_\kappa(z) t^{-1} 
\ = \ 
\lambda_\kappa(\chi(t)z)
\qquad
\text{ for all } 
t \in T, z \in \KK.
$$
Remark~\ref{rem:downgrade} and Lemma~\ref{lem:homogvsnormalize}
show that under the identification $\Chi(T) = \ZZ^s$ 
the character $\chi$ is just the $\ZZ^s$-part 
of the Demazure $P$-root $\kappa$.
Proposition~\ref{prop:DEM2deriv} 
tells us that any element of $\LND_{K_0}(R(A,P))_0$
is a linear combination of derivations 
$\delta_\kappa$ arising 
from Demazure $P$-roots.
Moreover, by Corollary~\ref{cor:1pasglift}, 
the push forward $p_*$ maps 
$\PASG_{H_0}(\b{G})$ onto $\PASG_{T}(G)$.
We conclude that $\Lie(G)$ is spanned as a
$\KK$-vector space  
by $\Lie(T)$ and $\dot \lambda_\kappa(0)$,
where $\kappa$ runs through the Demazure 
$P$-roots.
Assertion~(ii) follows.
\end{proof}

\begin{corollary}[of proof]
\label{cor:autXgen}
Let $X$ be a nontoric normal complete rational variety 
with a torus action $T \times X \to X$ of complexity one 
arising as a good quotient $p \colon \rq{X} \to X$ from
$R(A,P)$ according to Construction~\ref{constrRAP2TVar}.
\begin{enumerate}
\item
Every Demazure $P$-root $\kappa$ 
induces an additive one parameter subgroup 
$\lambda_{\kappa} = p_* \lambda_{\delta_\kappa} 
\colon \GG_a \to \Aut(X)$.
\item
The Demazure $P$-root $\kappa$ is vertical
if and only if the general orbit of $\lambda_{\kappa}$
is contained in some $T$-orbit closure.
\item
The Demazure $P$-root $\kappa$ is horizontal
if and only if the general orbit of $\lambda_{\kappa}$ 
is not contained in any $T$-orbit closure.
\item
The unit component $\Aut(X)^0$ of the automorphism 
group is generated by $T$ and the images 
$\lambda_{\kappa}(\GG_a)$.
\end{enumerate}
\end{corollary}

\begin{proof}
Assertions~(i) and~(iv) are clear by the proof of 
Theorem~\ref{thm:autroots}.
For~(ii) and~(iii) recall that $\kappa$ is 
vertical (horizontal) if and only if 
$\delta_{\kappa}$ is of vertical (horizontal)
type. 
The latter is equivalent to saying that  
$\lambda_{\kappa}(\GG_a)$ acts trivially (non-trivially)
on the field of $T$-invariant rational functions.
\end{proof}

\begin{example}[The $E_6$-singular cubic V]
Let $A$ and $P$ as in Example~\ref{ex:e6demProots}.
From there we infer that $R(A,P)$ admits precisely 
one horizontal Demazure $P$-root. 
For the automorphism group of the corresponding 
surface $X$ this means that $\Aut(X)^0$ is the 
semidirect product of $\KK^*$ and $\GG_a$ twisted 
via the weight~3, see again~\ref{ex:e6demProots}.
In particular, the surface X is almost homogeneous. 
Moreover, in this case, one can show directly that 
the group of graded automorphisms of $R(A,P)$ 
is connected. 
Thus, Theorem~\ref{thm:weakaut} yields that $\Aut(X)$ 
is the semidirect product of $\KK^*$ and $\GG_a$.
This is in accordance with~\cite{Sa}; we would like 
to thank Antonio Laface for mentioning this 
reference to us.
\end{example}

\section{Almost homogeneous surfaces}
\label{sec:delpezzo}

A variety is {\em almost homogeneous\/} if its 
automorphism group acts with an open orbit.
We take a closer look to this case with a special
emphasis on almost homogeneous rational 
$\KK^*$-surfaces of Picard number one. 
The first statement characterizes 
the almost homogeneous varieties coming with 
a torus action of complexity one in arbitrary 
dimension.

\begin{theorem}
\label{thm:almosthomogchar}
Let $X$ be a nontoric normal complete rational 
variety with a torus action $T \times X \to X$ 
of complexity one and Cox ring $\mathcal{R}(X) = R(A,P)$.
Then the following statements are equivalent.
\begin{enumerate}
\item
The variety $X$ is almost homogeneous.
\item
There exists a horizontal Demazure 
$P$-root.
\end{enumerate}
Moreover, if one of these statements holds, 
and $R(A,P)$ is minimally presented, then 
the number $r-1$ of relations of $R(A,P)$ is 
bounded by
\begin{eqnarray*}
r \ - \ 1 
& \le & 
\dim(X) \ + \ \rk(\Cl(X)) \ - \ m \ - \ 2.
\end{eqnarray*}
\end{theorem}

\begin{proof}
If~(i) holds, 
then $\Aut(X)$ acts with an open orbit 
on $X$ and by Corollary~\ref{cor:autXgen}, 
there must be a horizontal Demazure 
$P$-root $\kappa$. 
Conversely, if~(ii) holds, then there is 
a horizontal Demazure $P$-root $\kappa$ and
Corollary~\ref{cor:autXgen} says that for 
$U = p_*(\delta_\kappa(\GG_a))$, the group 
$T \ltimes U$ acts with an open orbit
on $X$.

For the supplement, recall first that $R(A,P)$ 
is a complete intersection with $r-1$ necessary 
relations and thus we have 
$$
n + m -  (r-1)
\ = \ 
\dim(R(A,P))
\ = \ 
\dim(X) + \rk(\Cl(X)).
$$
Now observe that any relation 
$g_I$ involving only three variables prevents
existence of a horizontal Demazure $P$-root.
Consequently, by suitably arranging the relations,
we have $n_0,n_1 \ge 1$ and $n_i \ge 2$ for all
$i \ge 2$.
Thus, $n \ge 2 + 2(r-1)$ holds and the assertion 
follows.
\end{proof}

We specialize to dimension two.
Any normal complete rational $\KK^*$-surface
$X$ is determined by its Cox
ring and thus is given up to isomorphism 
by the defining data $A$ and $P$ of the ring 
$\mathcal{R}(X) = R(A,P)$;
we also say that the $\KK^*$-surface $X$ arises 
from $A$ and $P$ and refer to~\cite[Sec.~3.3]{Ha3L} 
for more background.
A first step towards the almost homogeneous~$X$ 
is to determine possible horizontal Demazure 
$P$-roots in the following setting.

\begin{proposition}
\label{prop:demrootsexpl}
Consider integers $l_{02} \ge 1$, $l_{11} \ge l_{21} \ge 2$ 
and $d_{01}$, $d_{02}$, $d_{11}$, $d_{21}$
such that the following matrix has pairwise different 
primitive columns generating~$\QQ^3$ 
as a convex cone:
\begin{eqnarray*}
P 
& := &
\left[
\begin{array}{rrrr} 
-1 & -l_{02} & l_{11} & 0
\\
-1 & -l_{02} & 0 & l_{21}
\\
d_{01} & d_{02} & d_{11} & d_{21}
\end{array}
\right].
\end{eqnarray*}
Moreover, assume that $P$ is positive in the sense that 
$\det(P_{01}) > 0$ holds, where $P_{01}$ is the $3 \times 3$
matrix obtained from $P$
by deleting the first column.
Then the possible horizontal Demazure $P$-roots are 
\begin{enumerate}
\item
$\kappa = (u,1,2,(1,1,1))$, where 
$u =   
\left(
d_{01} \alpha
+ \frac{d_{21}\alpha + 1}{l_{21}}
\, ,
- \frac{d_{21}\alpha + 1}{l_{21}}
\, ,\alpha
\right)
$
with an integer $\alpha$ satisfying
$$
\qquad\quad
l_{21}
\mid 
d_{21} \alpha + 1,
\qquad
\frac{l_{02}}{d_{02}-l_{02}d_{01}}
\ \le \ 
\alpha
\ \le \ 
-\frac{l_{11}}{l_{21}d_{11}+l_{11}d_{21}+d_{01}l_{11}l_{21}},
$$
\item
if $l_{02} = 1$: $\kappa = (u,1,2,(2,1,1))$, where 
$u =   
\left(
d_{02} \alpha
+ \frac{d_{21}\alpha + 1}{l_{21}}
\, ,
- \frac{d_{21}\alpha + 1}{l_{21}}
\, ,\alpha
\right)
$
with an integer $\alpha$ satisfying
$$
\qquad\quad
l_{21}
\mid 
d_{21} \alpha + 1,
\qquad
-\frac{l_{11}}{l_{21}d_{11}+l_{11}d_{21}+d_{02}l_{11}l_{21}}
\ \le \
\alpha
\ \le \ 
\frac{1}{d_{01}-d_{02}},
$$
\item
$\kappa = (u,2,1,(1,1,1))$, where 
$u =   
\left(
- \frac{d_{11}\alpha + 1}{l_{11}}
\, ,
d_{01}\alpha + \frac{d_{11}\alpha + 1}{l_{11}}
\, ,\alpha
\right)
$
with an integer $\alpha$ satisfying
$$
\qquad\quad
l_{11}
\mid 
d_{11} \alpha + 1,
\qquad
\frac{l_{02}}{d_{02}-l_{02}d_{01}}
\ \le \
\alpha
\ \le \ 
-\frac{l_{21}}{l_{21}d_{11}+l_{11}d_{21}+d_{01}l_{11}l_{21}},
$$
\item
if $l_{02} = 1$: $\kappa = (u,2,1,(2,1,1))$, where 
$u =   
\left(
- \frac{d_{11}\alpha + 1}{l_{11}}
\, ,
d_{02}\alpha + \frac{d_{11}\alpha + 1}{l_{11}}
\, ,\alpha
\right)
$
with an integer $\alpha$ satisfying
$$
\qquad\quad
l_{11}
\mid 
d_{11} \alpha + 1,
\qquad
-\frac{l_{21}}{l_{21}d_{11}+l_{11}d_{21}+d_{02}l_{11}l_{21}}
\ \le \
\alpha
\ \le \ 
\frac{1}{d_{01}-d_{02}}.
$$
\end{enumerate}
\end{proposition}

\begin{proof}
In the situation of~(i), evaluating the 
general linear form 
$u = (u_1,u_2,u_3)$ on the columns of $P$ 
gives the following conditions for a 
Demazure $P$-root:
$$
-u_1 - u_2 + u_3 d_{01} = 0,
\qquad
u_2 l_{21} + u_3 d_{21} = -1,
$$
$$
-u_1 l_{02} - u_2 l_{02} + u_3 d_{02} \ge l_{02},
\qquad
u_1 l_{11} + u_3 d_{11} \ge 0.
$$
Resolving the equations for $u_1,u_2$ and plugging
the result into the  inequalities gives the
desired roots with $\alpha := u_3$. 
The other cases are treated analogously.
\end{proof}

\begin{corollary}
\label{thm:alhomdelp}
The  nontoric  almost homogeneous normal complete 
rational $\KK^*$-surfaces $X$ of Picard number one 
are precisely the ones arising from 
data
$$ 
A 
\ = \ 
\left[
\begin{array}{ccc} 
0 & -1 & 1
\\
1 & -1 & 0 
\end{array}
\right],
\qquad\qquad
P \ = \ 
\left[
\begin{array}{rrrr} 
-1 & -l_{02} & l_{11} & 0
\\
-1 & -l_{02} & 0 & l_{21}
\\
d_{01} & d_{02} & d_{11} & d_{21}
\end{array}
\right]
$$
as in~\ref{prop:demrootsexpl} allowing an integer $\alpha$ 
according to one
of the Conditions~\ref{prop:demrootsexpl}~(i) to~(iv). 
In particular, the Cox ring of $X$ is given as 
\begin{eqnarray*}
\mathcal{R}(X)
& = & 
\KK[T_{01},T_{02},T_{11},T_{21}] 
\ / \ 
\bangle{T_{01}T_{02}^{l_{02}} + T_{11}^{l_{11}} + T_{21}^{l_{21}}}
\end{eqnarray*}
with the grading by $\ZZ^4 / {\rm im}(P^*)$.
Moreover, the anticanonical divisor of $X$ is ample, 
i.e.~$X$ is a del Pezzo surface.
\end{corollary}

\begin{proof}
As any surface with finitely generated Cox ring, 
$X$ is $\QQ$-factorial.
Since~$X$ has Picard number one,
the divisor class group 
$\Cl(X)$ is of rank one.
Now take a minimal presentation
$\mathcal{R}(X) = R(A,P)$ of the Cox ring.
Then, according to Theorem~\ref{thm:almosthomogchar}, 
we have $m=0$ and there is exactly one relation 
in $R(A,P)$.
Thus~$P$ is a $3 \times 4$ matrix.
Moreover, Theorem~\ref{thm:almosthomogchar} says that 
there is a horizontal Demazure $P$-root.
Consequently, one of the exponents $l_{01}$ and $l_{02}$ 
must equal one, say $l_{01}$.
Fixing a suitable order for the last two variables
we ensure $l_{11} \ge l_{21}$.
Passing to the $\KK^*$-action $t^{-1} \cdot x$ 
instead of $t \cdot x$ if necessary, we achieve that
$P$ is positive in the sense of~\ref{prop:demrootsexpl}.

Let us see why $X$ is a del Pezzo surface.
Denote by $P_{ij}$ the 
matrix obtained from $P$ by deleting the column 
$v_{ij}$. 
Then, in $\Cl(X)^0 = \ZZ$, the factor group of 
$\Cl(X)$ by the torsion part, the weights 
$w_{ij}^0$ of $T_{ij}$ are given up to a factor 
$\alpha$ as
\begin{eqnarray*}
(w_{01}^0,w_{02}^0,w_{11}^0,w_{21}^0)
& = & 
\alpha (\det(P_{01}),-\det(P_{02}),\det(P_{11}),-\det(P_{21})).
\end{eqnarray*}
According to~\cite[Prop.~III.3.4.1]{ArDeHaLa}, 
the class of the 
anticanonical divisor in $\Cl(X)^0$ is given 
as the sum over all $w_{ij}^0$ minus 
the degree of the relation.
The inequalities on the $l_{ij},d_{ij}$ 
implied by the existence of 
an integer $\alpha$ as in~\ref{prop:demrootsexpl} ~(i) to~(iv)
show that the anticanonical class is positive
(note that $\alpha$ rules out). 
\end{proof}

We turn to the case of precisely one singular point.
The diophantic aspect of  
Conditions~\ref{prop:demrootsexpl}~(i) to~(iv)
on existence of Demazure $P$-roots then disappears:
there is no divisibility condition any more.

\begin{construction}[$\KK^*$-surfaces with one singularity]
Consider a triple $(l_0,l_1,l_2)$ of integers 
satisfying the following conditions:
$$
l_0 \ \ge \ 1, 
\qquad
l_1 \ \ge \ l_2 \ \ge \ 2,
\qquad
l_0 \ < \ l_1l_2,
\qquad
\gcd(l_1,l_2) \ = \ 1.
$$
Let $(d_1,d_2) $ be the (unique) pair of integers with 
$d_1l_2+d_2l_1 = -1$ and $0 \le d_2 < l_2$ and consider
the data
$$ 
A 
\ = \ 
\left[
\begin{array}{ccc} 
0 & -1 & 1
\\
1 & -1 & 0 
\end{array}
\right],
\qquad\qquad
P \ = \ 
\left[
\begin{array}{rrrr} 
-1 & -l_{0} & l_{1} & 0
\\
-1 & -l_{0} & 0 & l_{2}
\\
0 &  1 & d_{1} & d_{2}
\end{array}
\right]
$$
Then the associated ring $R(l_0,l_1,l_2) := R(A,P)$ 
is graded by $\ZZ^4 / {\rm im}(P^*) \cong \ZZ$,
and is explicitly given by
$$ 
R(l_0,l_1,l_2)
\ = \ 
\KK[T_{1},T_{2},T_{3},T_{4}] / \bangle{T_1T_2^{l_0} +  T_3^{l_1} + T_4^{l_2}},
$$
$$ 
\deg(T_1) = l_1l_2 - l_0, 
\qquad
\deg(T_2) = 1, 
\qquad
\deg(T_3) = l_2, 
\qquad
\deg(T_4) = l_1.
$$
\end{construction}

\begin{proposition}
\label{prop:onesing}
For the $\KK^*$-surface $X = X(l_0,l_1,l_2)$ 
with Cox ring $R(l_0,l_1,l_2)$, the following 
statements hold:
\begin{enumerate}
\item
$X$ is nontoric and we have $\Cl(X) = \ZZ$,
\item
$X$ comes with precisely 
one singularity,
\item
$X$ is a del Pezzo surface if and only 
if $l_0 < l_1+l_2+1$ holds,
\item 
$X$ is almost homogeneous if and only 
if $l_0 \le l_1$ holds.
\end{enumerate}
Moreover, any normal complete rational 
nontoric $\KK^*$-surface of Picard number 
one with precisely one singularity is 
isomorphic to some $X(l_0,l_1,l_2)$.
\end{proposition}

\begin{proof}
First note that $X = X(l_0,l_1,l_2)$ is obtained as 
in Construction~\ref{constrRAP2TVar}: the group 
$H_X = \KK^*$ acts on $\KK^4$ by 
\begin{eqnarray*}
t \cdot z 
& = & 
(t^{l_1l_2 - l_0}z_1, tz_2, t^{l_2}z_3, t^{l_1}z_4),
\end{eqnarray*}
the total coordinate space
$\b{X} := V(T_1T_2^{l_0} +  T_3^{l_1} + T_4^{l_2})$
is invariant under this action
and we have  
$$
\rq{X} \ = \ \b{X} \setminus \{0\},
\qquad\qquad
X \ = \ \rq{X} / \KK^*.
$$
Thus, $\Cl(X) = \ZZ$ holds and, since the Cox ring 
$\mathcal{R}(X) = R(l_0,l_1,l_2)$ is not a polynomial 
ring, $X$ is nontoric.

Using~\cite[Prop.~III.3.1.5]{ArDeHaLa}, we show 
that the sest of singular points of $X$ 
consists of the image $x_0 \in X$ of the point 
$(1,0,0,0) \in \rq{X}$ under the quotient map 
$\rq{X} \to X$.
If $l_1l_2 - l_0 > 1 $ holds,
then the local divisor class group 
\begin{eqnarray*}
\Cl(X,x_0)
& = & 
\ZZ / (l_1l_2 - l_0) \ZZ
\end{eqnarray*}
is nontrivial and thus $x_0 \in X$ singular.
If $l_1l_2 - l_0 = 1 $ holds, then we have 
$l_0 > 1$ and therefore $(1,0,0,0) \in \rq{X}$
and hence $x_0 \in X$ is singular.
Since all other local divisor class groups of $X$ 
are trivial and, moreover, all singular points of $\rq{X}$ 
lie in the orbit $\KK^* \cdot (1,0,0,0)$, 
we conclude that $x_0 \in X$ is 
the only singular point.

According to~\cite[Prop.~III.3.4.1]{ArDeHaLa}, 
the anticanonical class of $X$ is 
$l_1+l_2+1 -l_0$. This proves~(iii).
Finally, for~(iv), we infer from 
Proposition~\ref{prop:demrootsexpl}
that existence of a horizontal Demazure 
$P$-root is equivalent to existence of 
an integer $\alpha$ with 
$l_0 \le \alpha \le l_1$ which in turn is equivalent 
to $l_0 \le l_1$.

We come to the supplement. The surface $X$ arises 
from a ring $R(A,P)$, where we may assume that 
$R(A,P)$ is minimally presented. The first task is 
to show that $n=4$, $m=0$ and $r=2$ holds.
We have
$$ 
n  + m - (r-1)
\ = \ 
\dim(X) + \rk(\Cl(X))
\ = \ 
3.
$$
Any relation $g_I$ involving only three variables
gives rise to a singularity in the source and 
a singularity in the sink of the $\KK^*$-action. 
We conclude that at most two of the monomials 
occuring in the relations may depend only on one 
variable. Thus, the above equation shows that
$n=4$, $m=0$ and $r=2$ hold.

We may assume that the defining equation is of 
the form 
$T_{01}^{l_{01}}T_{02}^{l_{02}} + T_{11}^{l_{11}} + T_{21}^{l_{21}}$.
Again, since one of the two elliptic fixed 
points must be smooth, we can conclude that one 
of $l_{0j}$ equals one, say $l_{01}$.
Now it is a direct consequence of the 
description of the local divisor class groups 
given in~\cite[Prop.~III.3.1.5]{ArDeHaLa}
that a $\KK^*$-surface with precisely one 
singularity arises from a matrix $P$ as in the 
assertion.
\end{proof}

Now we look at the log terminal ones of the 
$X(l_0,l_1,l_2)$; recall, that a singularity 
is {\em log terminal\/} if all its resolutions 
have discrepancies bigger than $-1$. Over $\CC$,
the log terminal surface singularities are 
precisely the quotient singularities by subgroups 
of $\GL_2(\CC)$, see for example~\cite[Sec.~4.6]{Mat}.
The {\em Gorenstein index\/} of $X$ is the 
minimal positive integer $\imath(X)$ such 
that $\imath(X)$ times the canonical divisor
$\mathcal{K}_X$ is Cartier.

\begin{corollary}
\label{cor:logdelp}
Assume that $X = X(l_1,l_2,l_3)$ is log terminal.
Then we have the following three cases:
\begin{enumerate}
\item
the surface $X$ is almost homogeneous,
\item 
the singularity of $X$ is of type $E_7$,
\item 
the singularity of $X$ is of type $E_8$.
\end{enumerate}
Moreover, for the almost homogenoeus 
surfaces $X = X(l_1,l_2,l_3)$
of Gorenstein index $\imath(X) = a$, we have 
\begin{enumerate}
\item
$(l_0,l_1,l_2) = (1,l_1,l_2)$ 
with the bounds 
$l_2 \le l_1 \le 2a^2 + \frac{4}{3}a$,
\item
$(l_0,l_1,l_2) = (2,l_1,2)$ 
with the bound $l_1 \le 3a+2$,
\item
$(l_0,l_1,l_2) = (3,3,2), (2,4,3), (2,5,3), (3,5,2)$.
\end{enumerate}
\end{corollary}

\begin{proof}
The condition that $X$ is log terminal means that 
the number $l_0l_1l_2$ is bounded by $l_0l_1+l_0l_2+l_1l_2$;
this can be seen by explicitly performing the 
canonical resolution of 
singularities of $X(l_0,l_1,l_2)$, 
see~\cite[Sec.~3]{Ha3L}. 
Thus, the allowed $(l_0,l_1,l_2)$ must be 
platonic triples and we are left with
$$ 
(1,l_1,l_2), 
\
(2,l_1,2),
\
(3,3,2),
\
(2,4,3),
\
(2,5,3),
\
(3,5,2),
\
(4,3,2),
\
(5,3,2).
$$
The last two give the surfaces with singularities 
$E_7$, $E_8$ and in all other cases, the resulting 
surface is almost homogeneous by 
Proposition~\ref{prop:onesing}.
The Gorenstein condition says that 
$a \mathcal{K}_X$ 
lies in the Picard group.
According to~\cite[Cor.~III.3.1.6]{ArDeHaLa}, 
this is equivalent to the fact that $l_1l_2 -l_0$ 
divides $a \cdot (l_1+l_2+1-l_0)$.
The bounds then follow by elementary estimations.
\end{proof}

\begin{corollary}
The following tables list the triples 
$(l_0,l_1,l_2)$ together with roots 
of $\Aut(X)$ for the log terminal
almost homogeneous complete rational 
$\KK^*$-surfaces $X = X(l_0,l_1,l_2)$
with precisely one 
singularity up to Gorenstein index 
$\imath(X) = 5$.

\bigskip

\begin{center}
\begin{tabular}{l|l|l}
$\imath(X) \ = \ 1$  
& 
$\imath(X) \ = \ 2$ 
& 
$\imath(X) \ = \ 3$ 
\\
\hline
$(1,3,2):\{1,2,3\}$
&
$(1,7,3):\{1,3,4,7\}$
&
$(2,7,2):\{2,3,5,7\}$
\\
$(2,3,2):\{2,3\}$
&
&
$(1,13,4):\{1,4,5,9,13\}$
\\
$(3,3,2):\{3\}$
&
&
$(1,8,5):\{3,5,8\}$
\end{tabular}
\end{center}

\bigskip

\begin{center}
\begin{tabular}{l|l}
$\imath(X) \ = \ 4$ 
& 
$\imath(X) \ = \ 5$
\\
\hline
$(2,5,2):\{2,3,5\}$
&
$(2,11,2):\{2,3,5,7,9,11\}$
\\
$(1,21,5):\{1,5,6,11,16,21\}$
&
$(1,13,7):\{2,6,13\}$
\\
&
$(2,4,3):\{3,4\}$
\\
&
$(1,17,3):\{2,3,5,8,11,14,17\}$
\\
&
$(1,31,6):\{1,6,7,13,19,25,31\}$
\\
&
$(1,18,7):\{4,7,11,18\}$
\end{tabular}
\end{center}
\end{corollary}

\section{Structure of the semisimple part}
\label{sec:semi}

We describe the root system of the semisimple part 
of the automorphism group of a nontoric normal 
complete rational variety admitting a torus action 
of complexity one. 
Let us first recall the necessary background
on semisimple groups and their root systems.

A connected linear algebraic group $G$ is semisimple
if it has only trivial closed connected commutative 
normal subgroups. Any linear algebraic group $G$ 
admits a maximal connected semisimple subgroup
$G^{\rm ss} \subseteq G$ called a semisimple part.
The semisimple part is unique up to conjugation 
by elements from the unipotent radical. 
If $G$ is semisimple, then the set 
$\Phi_G \subseteq \Chi_{\RR}(T)$ of roots 
with respect to a given maximal torus $T \subseteq G$ 
is a root system, i.e.~for every $\alpha \in  \Phi_G$
one has 
$$
\Phi_G \cap \RR \alpha \ = \ \{\pm \alpha \},
\qquad\qquad
s_\alpha(\Phi_G) \ = \ \Phi_G,
$$
where $s_\alpha \colon \Chi_{\RR}(T) \to \Chi_{\RR}(T)$
denotes the reflection with fixed hyperplane 
$\alpha^\perp$ with respect to a given scalar 
product on $\Chi_{\RR}(T)$.
The possible root systems are elementarily classified; 
for us, the following types (always realized with the 
standard scalar product) will be important:
$$ 
A_n 
\ := \ 
\{e_i-e_j; \; 1 \le i,j \le n+1, \, i \ne j\}
\ \subseteq \ 
\RR^{n+1},
$$
$$ 
B_2 
\ := \ 
\{\pm e_1, \, \pm e_2, \, \pm(e_1+e_2), \, \pm(e_1-e_2)\}
\ \subseteq \ 
\RR^2.
$$
The root system of a connected semisimple linear group $G$ 
determines $G$ up to coverings;
for example, $A_n$ belongs to the (simply connected) 
special linear group~$\SL_{n+1}$ and $B_2$ to the 
(simply connected) symplectic group $\Sp_4$.

We turn to varieties with a complexity 
one torus action.
Consider data $A,P$ as in 
Construction~\ref{constr:RAPdown}
and the resulting ring $R(A,P)$ 
with its fine $K_0$-grading and the 
coarser $K$-grading. Recall that 
the fine grading group $K_0$ splits
as 
$$ 
K_0 \ = \ K_0^{\rm vert} \oplus K_0^{\rm hor},
\quad 
\text{where}
\quad
K_0^{\rm vert} \ := \ \bangle{\deg_{K_0}(S_k)},
\quad
K_0^{\rm hor} \ := \ \bangle{\deg_{K_0}(T_{ij})}.
$$
Note that $K_0^{\rm vert} \cong \ZZ^m$ is freely 
generated by $\deg_{K_0}(S_1), \ldots, \deg_{K_0}(S_m)$. 
Moreover, by Remark~\ref{rem:downgrade},
the direct factor $\ZZ^s$ of the column 
space $\ZZ^{r+s}$ of $P$ is identified 
via $Q_0 \circ P^*$ with the kernel of the 
downgrading map $K_0 \to K$.

\begin{definition}
Let $A,P$ be as in Construction~\ref{constr:RAPdown}
such that the associated ring $R(A,P)$ is minimally 
presented and write $\alpha_\kappa$ for the $P$-root, 
i.e.~the $\ZZ^s$-part,
associated to Demazure $P$-root~$\kappa$.
\begin{enumerate}
\item
We call a $P$-root $\alpha_\kappa$ {\em semisimple\/} if 
$-\alpha_{\kappa} = \alpha_{\kappa'}$ holds
for some Demazure $P$-root~$\kappa'$.
\item
We call a semisimple $P$-root $\alpha_\kappa$ 
{\em vertical\/} if 
$\alpha_{\kappa} \in K_0^{\rm vert}$
and 
{\em horizontal\/} if 
$\alpha_{\kappa} \in K_0^{\rm hor}$ holds.
\item 
We write $\Phi_P^{\rm ss}$, $\Phi_P^{\rm vert}$ 
and $\Phi_P^{\rm hor}$ for the set of semisimple, 
vertical semisimple and horizontal semisimple 
$P$-roots in $\RR^s$ respectively.
\end{enumerate}
\end{definition}

\begin{theorem}
\label{thm:sesipart}
Let $A,P$ be as in Construction~\ref{constr:RAPdown}
such that $R(A,P)$ is minimally presented and 
let $X$ be a (nontoric) variety with a complexity one 
torus action $T \times X \to X$ arising from
$A,P$ according to Construction~\ref{constrRAP2TVar}.
Then the following statements hold.
\begin{enumerate}
\item
$\Phi_P^{\rm vert}, \Phi_P^{\rm hor}$ and $\Phi_P^{\rm ss}$ 
are root systems, we have 
$\Phi_P^{\rm ss} =  \Phi_P^{\rm vert}  \oplus  \Phi_P^{\rm hor}$
and $\Phi_P^{\rm ss}$ is the root system with respect 
to $T$ of the semisimple part $\Aut(X)^{\rm ss}$.
\item
For $p \in K$ denote by $m_p$ the number 
of variables $S_k$ with $\deg_K(S_k) = p$.
Then we have 
$$ 
\qquad\qquad
\Phi_P^{\rm vert} \ \cong \ \bigoplus_{p \in K} A_{m_p -1},
\qquad
\qquad
\sum_{p \in K} (m_p-1) \ < \ \dim(X) - 1.
$$
\item
Suppose $\Phi_P^{\rm hor} \ne \emptyset$. 
Then $r=2$ holds, and, after suitably renumbering 
the variables one has
\begin{enumerate}
\item
$T_{01}T_{02}+T_{11}T_{12}+T_{2}^{l_{2}}$,
$\b{w}_{01}=\b{w}_{11}$
and
$\b{w}_{02}=\b{w}_{12}$,
\item
$T_{01}T_{02}+T_{11}^2+T_{2}^{l_{2}}$,
and
$\b{w}_{01}=\b{w}_{02}=\b{w}_{11}$
\end{enumerate}
for the defining relation of $R(A,P)$ 
and the degrees $\b{w}_{ij} = \deg_K(T_{ij})$
of the variables.
\item
In the above case (iii\,a), we obtain the following possibilities 
for the root system $\Phi_P^{\rm hor}$: 
\begin{itemize}
\item
If $l_{21} + \ldots + l_{2n_2} \ge 3$ holds, then one has
$$
\Phi_P^{\rm hor}
\ = \ 
\begin{cases}
A_1 \oplus A_1,
&  
\b{w}_{01}=\b{w}_{02}=\b{w}_{11}=\b{w}_{12},
\\
A_1, 
& 
\text{otherwise}.
\end{cases}
$$
\item
If $n_2=2$ and $l_{21}=l_{22} = 1$ hold, then one has
$$
\hspace*{2.1cm}
\Phi_P^{\rm hor}
\ = \ 
\begin{cases}
A_3,
&
\b{w}_{01}=\b{w}_{02}=\b{w}_{11}=\b{w}_{12}=\b{w}_{21}=\b{w}_{22},
\\
A_2,
&
\b{w}_{01}=\b{w}_{11}=\b{w}_{21}, 
\b{w}_{02}=\b{w}_{12}=\b{w}_{22},
\b{w}_{01} \ne \b{w}_{02},
\\
A_1 \oplus A_1,
&  
\b{w}_{01}=\b{w}_{02}=\b{w}_{11}=\b{w}_{21},
\b{w}_{01} \ne \b{w}_{21},
\b{w}_{01} \ne \b{w}_{22},
\\
A_1, 
& 
\text{otherwise}.
\end{cases}
$$
\end{itemize}
\item
In the above case (iii\,b), we obtain the following possibilities 
for the root system $\Phi_P^{\rm hor}$: 
\begin{itemize}
\item
If $l_{21} + \ldots + l_{2n_2} \ge 3$ holds, then one has
$$
\Phi_P^{\rm hor}
\ = \ 
A_1. 
$$
\item
If $n_2=1$ and $l_{21} =2$ hold, then one has
$$
\Phi_P^{\rm hor}
\ = \ 
\begin{cases}
A_1 \oplus A_1,
&  
\b{w}_{01}=\b{w}_{02}=\b{w}_{11} = \b{w}_{21},
\\
A_1, 
& 
\text{otherwise}.
\end{cases}
$$
\item
If $n_2=2$ and $l_{21}=l_{22} = 1$ hold, then one has
$$
\Phi_P^{\rm hor}
\ = \ 
\begin{cases}
B_2,
&  
\b{w}_{01}=\b{w}_{02}=\b{w}_{11}=\b{w}_{21} = \b{w}_{22},
\\
A_1, 
& 
\text{otherwise}.
\end{cases}
$$
\end{itemize}
\end{enumerate}
\end{theorem}

The rest of the section is devoted to the proof of this 
theorem. 
Some of the steps are needed later on and 
therefore formulated as Lemmas.
We fix $A,P$ as in Construction~\ref{constr:RAPdown}
and assume that $R(A,P)$ is minimally presented.

\begin{lemma}
\label{lem:lndweightdecomp}
Let $\delta$ be a nonzero primitive $K_0$-homogeneous 
locally nilpotent derivation on~$R(A,P)$, decompose
$\deg_{K_0}(\delta) = w^{\rm vert}+w^{\rm hor}$
according to $K_0 = K_0^{\rm vert} \oplus K_0^{\rm hor}$ 
and write $w^{\rm vert} = \sum b_k w_k$ with 
$w_k = \deg_{K_0}(S_k)$ and $b_k \in \ZZ$.
\begin{enumerate}
\item
If $\delta$ is of vertical type, then there is a $k_0$ 
with $b_{k_0} = -1$ and $b_k \ge 0$ for all $k \ne k_0$.
Moreover, $w^{\rm hor}$ belongs to the weight monoid 
of $R(A,P)$.
\item
If $\delta$ is of horizontal type, then $b_k \ge 0$ 
holds for all $k$.
Moreover, $-w^{\rm hor}$ does not belong to the weight 
monoid of $R(A,P)$.
\end{enumerate}
\end{lemma}

\begin{proof}
If $\delta$ is nonzero of vertical type, then 
Theorem~\ref{thm:basiclnd}~(i) directly 
yields the assertion.
If $\delta$ is nonzero of horizontal type, then 
it is of the form $\delta = h \delta_{C,\beta}$
as in Theorem~\ref{thm:basiclnd}~(ii).
For the degree of $\delta$, we have 
$$ 
\deg_{K_0}(\delta)
\ = \ 
w^{\rm vert} + w^{\rm hor}
\ = \ 
\deg_{K_0}(h)^{\rm vert} + \deg_{K_0}(h)^{\rm hor} + 
\deg_{K_0}(\delta_{C,\beta}),
$$
where $\deg_{K_0}(\delta_{C,\beta})$
lies in $K_0^{\rm hor}$ by
Construction~\ref{constr:phlnd}.
Since $\deg_{K_0}(h)^{\rm vert}$ belongs to the 
weight monoid $S$ and, due to primitivity,
$\deg_{K_0}(\delta)$ lies outside the weight 
cone, we must have $w^{\rm hor} \ne 0$.
Moreover, for a $T_{ij}$ with
$\delta(T_{ij}) \ne 0$
and $w_{ij} := \deg_{K_0}(T_{ij})$, 
the degree computation of 
Construction~\ref{constr:phlnd}
shows
$$
0 
\ \ne \
\deg_{K_0}(h)^{\rm hor} 
+ 
\deg_{K_0}(\delta_{C,\beta}) 
+ 
w_{ij}
\ = \ 
\deg_{K_0}(\delta(T_{ij}))^{\rm hor}.
$$
Thus, $w^{\rm hor} + w_{ij}$ is 
a nonzero element in $S$.
If also $-w^{\rm hor}$ belongs to $S$,
then we can take $0 \ne f,g \in R(A,P)$ 
homogeneous of degree $w^{\rm hor} + w_{ij}$
and $-w^{\rm hor}$ respectively.
The product $fg$ is of degree $w_{ij}$.
By Corollary~\ref{cor:gencompdim}, 
this means $fg = cT_{ij}$ with $c \in \KK$.
A contradiction to the fact that $T_{ij}$ 
is $K_0$-prime.
\end{proof}

By definition, the semisimple roots occur in 
pairs $\alpha_+,\alpha_-$ with $\alpha_++\alpha_-=0$.
Given a pair $\kappa_+, \kappa_-$ of Demazure
$P$-roots defining $\alpha_+,\alpha_-$,  write 
$\delta_+,\delta_-$ for the derivations arising
from $\kappa_+, \kappa_-$ via 
Construction~\ref{constr:DEMLND}.
We call the pairs $\kappa_+, \kappa_-$ and 
$\delta_+,\delta_-$ {\em associated\/}
to $\alpha_+,\alpha_-$.
Note that associated pairs $\kappa_+, \kappa_-$ or 
$\delta_+,\delta_-$ are in general 
not uniquely determined by $\alpha_+,\alpha_-$.

\begin{lemma}
\label{lem:root2nld}
Let $\delta_+,\delta_-$ be a pair of  primitive 
$K_0$-homogeneous locally nilpotent derivations 
associated to a pair $\alpha_+,\alpha_-$ of 
semisimple roots.
\begin{enumerate}
\item
The roots $\alpha_+$, $\alpha_-$ are the $K_0$-degrees 
of the derivations $\delta_+,\delta_-$. 
In particular, we have 
$$
\deg_{K_0}(\delta_+) + \deg_{K_0}(\delta_-)
\ = \ 
\alpha_+ + \alpha_-
\ = \ 
0.
$$
\item
If $\delta_+$ is of vertical type, then 
also $\delta_-$ is of vertical type and 
$\alpha_+,\alpha_-$ are both vertical.
\item
If $\delta_+$ is of horizontal type, then 
also $\delta_-$ is of horizontal type and 
$\alpha_+,\alpha_-$ are both horizontal.
\end{enumerate}
\end{lemma}

\begin{proof}
The first assertion is clear by
Construction~\ref{constr:DEMLND}.
Using the decomposition 
$K_0 = K_0^{\rm vert} \oplus K_0^{\rm hor}$,
we obtain
$$
\alpha_+^{\rm vert} + \alpha_-^{\rm vert}
\ = \ 0,
\qquad
\alpha_+^{\rm hor} + \alpha_-^{\rm hor}
\ = \ 0.
$$

If $\delta_+$ is of horizontal type, then 
Lemma~\ref{lem:lndweightdecomp}~(ii) shows
that $-\alpha_+^{\rm hor} = \alpha_-^{\rm hor}$
does not belong to the weight monoid.
Thus Lemma~\ref{lem:lndweightdecomp}~(i) 
says that $\delta_-$ must be of horizontal
type.
Moreover, Lemma~\ref{lem:lndweightdecomp}~(ii)
shows that $\alpha_+^{\rm vert}$ and 
$\alpha_-^{\rm vert}$ vanish and thus 
$\alpha_+,\alpha_-$ are horizontal.

If $\delta_+$ is of vertical type, 
then, by the preceeding consideration,
also $\delta_-$ must be vertical.
Moreover,  Lemma~\ref{lem:lndweightdecomp}~(i)
tells us that 
$\alpha_+^{\rm hor}$ and $\alpha_-^{\rm hor}$
both belong to the weight monoid.
As seen above they have opposite signs.
Since the weight cone is pointed,
we obtain $\alpha_+^{\rm hor} = \alpha_-^{\rm hor} = 0$
which means that $\alpha_+, \alpha_-$
are vertical.
\end{proof}

For the subsequent study, we will
perform certain elementary column 
and row operations with the matrix~$P$
which we will call {\em admissible\/}:
\begin{enumerate}[(I)]
\item
swap two columns inside a block 
$v_{ij_1}, \ldots, v_{ij_{n_i}}$,
\item
swap two whole column blocks 
$v_{ij_1}, \ldots, v_{ij_{n_i}}$ 
and $v_{i'j_1}, \ldots, v_{i'j_{n_{i'}}}$,
\item
add multiples of the upper $r$ rows
to one of the last $s$ rows,
\item 
any elementary row operation among the last $s$ 
rows.
\item
swap two columns inside the $d'$ block. 
\end{enumerate}
The operations of type (III) and (IV) do not change 
the ring $R(A,P)$ whereas the types (I), (II), (V)
correspond to certain renumberings of the variables 
of $R(A,P)$ keeping the (graded) isomorphy type.

For a Demazure $P$-root $\kappa = (u,k_0)$ of 
vertical type, the index $k_0$ is uniquely 
determined by the $\ZZ^s$-part of $u$.
Thus, we may speak of the {\em distinguished\/} 
index of a vertical $P$-root. 
Note that for any pair $\alpha^\pm$ of vertical 
semisimple $P$-roots, the distinguished indices 
satisfy $k_0^+ \ne k_0^-$.

\begin{lemma}
\label{lem:prephivert}
Let $1 \le k_0^+ < k_0^- \le m$ be given 
and denote by $f \in \ZZ^{n+m}$ the vector with  
$f_{k_0^\pm} = \mp 1$ and all other entries zero.
Then the following statements are equivalent.
\begin {enumerate}
\item
There exists a pair $\alpha^\pm$ of vertical 
semisimple roots with distinguished indices
$k_0^\pm$.
\item
The vector $f$ can be realized by
admissible operations of type~(III) and ~(IV)
as the $(r+1)$th row of $P$.
\item 
The variables $S_{k_0^+}$ and $S_{k_0^-}$ have the 
same $K$-degree.
\end{enumerate} 
\end{lemma}

\begin{proof} 
Suppose that (i) holds.
Let $\kappa^\pm = (u^\pm,k_0^\pm)$
be a pair of Demazure $P$-roots associated to $\alpha^\pm$.
Then $u := u^+ + u^-$ satisfies
$\bangle{u,v_{ij}} \ge 0$ for all $i,j$
and $\bangle{u,v_k} \ge 0$ for all $k$.
Since the columns of $P$ generate $\QQ^{r+s}$ as a 
cone, we obtain $u=0$.
Consequently $u^- = -u^+$ holds and we conclude
$$
\bangle{u^+,v_{ij}} = 0 \text{ for all } i,j,
\qquad
\bangle{u^+,v_k} = 0 \text{ for all } k \ne k_0^\pm,
$$
$$
\bangle{u^+,v_{k_0^+}} = -1, 
\qquad
\bangle{u^+,v_{k_0^-}} = 1.
$$
Now write $u^+ = (u_1^+,\alpha^+)$ with the $\ZZ^s$-part 
$\alpha^+$ and let $\sigma$ be an $(s-1) \times s$ matrix 
complementing the (primitive) row $\alpha^+$ to a 
unimodular matrix.
Consider the block matrix 
$$ 
\left[
\begin{array}{cc}
E_r & 0
\\
u_1^+ & \alpha^+
\\
0 & \sigma
\end{array}
\right].
$$
Applying this matrix from the left to $P$ describes
admissible operations of type~(III) and~(IV)
realizing the vector $f$ as the $(r+1)$th 
row of~$P$. Thus, (i) implies~(ii).

To see that (ii) implies~(i), we may assume that 
$f$ is already the $(r+1)$th row of~$P$.
Consider $u^\pm \in \ZZ^{r+s}$ having $u_{r+1}^{\pm} = \pm 1$ 
as the only nonzero coordinate. 
Then the $\ZZ^s$-parts $\alpha^\pm$ of the vertical 
Demazure $P$-roots $\kappa^\pm = (u^\pm,k_0^\pm)$     
are as wanted.

Clearly, (ii) implies (iii). Conversely, the implication 
``(iii)$\Rightarrow$(ii)'' is obtained by similar arguments 
as ``(i)$\Rightarrow$(ii)''.
\end{proof}

\begin{lemma}
\label{lem:phivert}
Let $\mathbb{E}_q$ denote the $q \times (q+1)$ block
matrix $[-\mathbf{1},E_q]$, where $\mathbf{1}$ is a 
column with all entries equal one and $E_q$ is the 
$q \times q$ unit matrix.  
After admissible operations of type (III), (IV) and (V),
the $[d,d']$ block of $P$ is of the form
\begin{eqnarray*}
[d,d']
& = & 
\left[
\begin{array}{ccccc}
0      & \mathbb{E}_{m_{p_1}} &        &       0           & 0
\\ 
\vdots &                  & \ddots &                   & \vdots
\\
0      &          0       &        &  \mathbb{E}_{m_{p_t}} & 0 
\\
d^*    & d''_1           &  \ldots &           d''_t       & d''
\end{array}
\right]
\end{eqnarray*}
where ${p_1}, \ldots, {p_t} \in K$ are the elements such that 
the number $m_{p_i}$ of variables $S_k$ of degree $p_i$ is 
at least two and $d_i''$ is a block of length $m_{p_i}$ with 
only the first column possibly nonzero. 
Moreover,  $\Phi_P^{\rm vert}$ 
is a root system and we have
$$ 
\Phi_P^{\rm vert} \ \cong \ \bigoplus_{p \in K} A_{m_p -1},
\qquad
\qquad
\sum_{p \in K} (m_p-1) \ < \ \dim(X) - 1.
$$
\end{lemma}

\begin{proof}
This is a direct application of Lemma~\ref{lem:prephivert}.
\end{proof}

\begin{lemma}
\label{lem:nlest}
If there exists a pair of semisimple 
roots $\alpha_\pm \in \Phi_P^{\rm hor}$,
then $r=2$ holds, and after suitably renumbering 
$l_0,l_1,l_2$, the following two cases remain
\begin{enumerate}
\item
We have $n_0 = n_1 = 2$ and $l_{01} = l_{02} = l_{11} = l_{12} = 1$
and for any pair $\delta_\pm$ of derivations associated to 
$\alpha_{\pm}$ one has $i_0^+ = i_0^- = 2$.
\item
We have  $n_0 = 1$, $l_{01} = 2$ and 
$n_1 = 2$, $l_{11} = l_{12} = 1$
and for any pair $\delta_\pm$ of derivations associated to 
$\alpha_{\pm}$ one has $i_0^+ = i_0^- = 2$.
\end{enumerate}
\end{lemma}

\begin{proof}
Lemma~\ref{lem:root2nld} says that $\delta_+,\delta_-$ 
are of horizontal type,
by Construction~\ref{constr:DEMLND} they are of the 
form $\delta_\pm = h_\pm \delta_{C^\pm,\beta^\pm}$ 
and the degree computation of 
Construction~\ref{constr:phlnd} gives
\begin{eqnarray*}
\deg_{K_0}(\delta_{\pm})
& = & 
(r-1)\mu \ - \ \sum_{k\ne i_0^{\pm}} w_{kc_k^{\pm}} \ + \ w_{\pm},
\end{eqnarray*}
where $\mu$ is the common $K_0$-degree of the relations,
$w_{ij}$ the $K_0$-degree of $T_{ij}$, the $i_0^\pm$-th 
component of $\beta_\pm$ vanishes and $w_\pm$ is the 
$K_0$-degree of $h_\pm$.
Now fix two distinct $i^+$, $i^-$ 
with $i^\pm\ne i_0^\pm$
and $l_{i^+c_{i^+}^+} = 1$.
Then 
$\deg_{K_0}(\delta_+) + \deg_{K_0}(\delta_-)=0$
leads to
$$
w
\ := \
w_+ \, + \, w_-
\, + \, 
\sum_{k\ne i_0^+, i^+} (\mu - w_{kc_k^+})  
\, + \, 
\sum_{k\ne i_0^-, i^-} (\mu - w_{kc_k^-})
\, = \, 
w_{i^+c_{i^+}^+}
\, + \,
w_{i^-c_{i^-}^-}.
$$
Note that all summands are elements of the weight monoid 
of $R(A,P)$ and, except possibly $w_\pm$, all are nonzero.
Moreover, Corollary~\ref{cor:gencompdim}~(ii) shows that 
the $K_0$-homogeneous component $R(A,P)_w$ is generated 
by  $f^+f^-$, where
$$
f^+ 
\ := \ 
T_{i^+c_{i^+}^+},
\qquad\qquad
f^-
\ := \ 
T_{i^-c_{i^-}^-}
.
$$
Now, choosing suitable presentations of the $\mu$,
write the first presentation of $w$ as the 
$K_0$-degree of a monomial $f$ in 
the variables $T_{ij}$ corresponding to the 
occuring~$w_{ij}$. Then $f = f^+f^-$ holds.
Since $f^+$ and $f^-$ are $K_0$-prime,
we conclude $w_\pm = 0$ and $r=2$.
Renumbering $i_0^+ \mapsto 2$, $i^+ \mapsto 1$, 
the above equation simplifies to
$$
w 
\ = \ 
\mu - w_{0c_0^+}
\, + \, 
\mu - w_{i_2^-c_{i_2^-}^-}
\ = \
 w_{1c_1^+}
\, + \,
w_{i^-c_{i^-}^-},
$$
where $i_2^-$ differs from $i_0^-$ and $i^-$.
We conclude further $i_2^- = 1$ and 
$\sum l_{0j} = \sum l_{1j} =2$
and $i_0^-=i_0^+$.
Since $l_{i^+c^+_{i^+}} = l_{1c^+_1}$ equals one,
we arrive at the cases~(i) and~(ii).
\end{proof}

The above lemma shows  that 
for a given pair $\alpha_\pm \in \Phi_P^{\rm hor}$,
all associated pairs of Demazure $P$-roots 
(or derivations) share the same 
$i_0 = i_0^+ = i_0^-$.
This allows us to speak about the 
{\em distinguished index $i_0$ of 
$\alpha_\pm \in \Phi_P^{\rm hor}$\/}.

\begin{lemma}
\label{lem:horss1}
Suppose $n_0 = n_1 = 2$, $l_{01} = l_{02} = l_{11} = l_{12} = 1$.
If there exists a pair $\alpha_\pm \in \Phi_P^{\rm hor}$
with distinguished index $i_0 = 2$, then $P$ can be brought 
by admissible operations, without moving the $n_2$-block,
into the form
\begin{eqnarray}
\label{glg:phor1} 
P
& = & 
\left[
\begin{array}{rrrrrr}
-1 & -1 & 1 & 1 & 0 & 0 
\\
-1 & -1 & 0 & 0 & l_2 & 0 
\\
-1 &  0     & 0   & 1     &  0 & 0
\\
0  &  0 & 0  & d_{12}^* & d_2^* & d'_* 
\end{array}
\right],
\end{eqnarray}
where the lower line is a matrix of size 
$(s-1) \times (n+m)$.
Conversely, if $P$ is of the above shape, 
then $\alpha_\pm = (\pm 1, 0) \in \Phi_P^{\rm hor}$ 
has distinguished index $i_0 = 2$.
Moreover, up to admissible operations of 
type~(III) and~(IV), situation~(\ref{glg:phor1}) 
is equivalent to
$$ 
\deg_K(T_{01}) \ = \ \deg_K(T_{12}),
\qquad\qquad
\deg_K(T_{02}) \ = \ \deg_K(T_{11}).
$$
\end{lemma}

\begin{proof}
Fix an associated pair $\kappa_\pm = (u^\pm,2,i_1^\pm,C_\pm)$ 
of Demazure $P$-roots.
Renumbering the variables, we first achieve $i_1^+ = 1$ 
and $C_+ = (1,1,1)$. 
Adding suitable multiples of the top two rows of $P$ to 
the lower $s$ rows, brings $P$ into the form
$$
P 
\ = \ 
\left[
\begin{array}{rrrrrr}
-1 & -1 & 1 & 1 & 0 & 0 
\\
-1 & -1 & 0 & 0 & l_2 & 0 
\\
0  &  d_{02} & 0  & d_{12} & d_2 & d' 
\end{array}
\right]
$$
Now we explicitly go through the defining conditions of 
the Demazure $P$-root $\kappa_+$ with 
$$
u^+ \ = \ (u_1^+,u_2^+,\alpha_+), \text{ where } u_i^\pm \in \ZZ, 
\qquad 
i_1^+ \ = \ 1,
\qquad
C_+ \ = \ (1,1,1).
$$ 
This gives in particular,
$u_1^+ = -1$ and $u_1^- = 1$.
Going through the conditions of a Demazure 
$P$-root $\kappa_- = (u^-,2,i_1^-,C_-)$ 
leaves us with the two possibilities
$$ 
u^- \ = \ (1,-1,-\alpha_+),
\qquad
i_1^- \ = \ 0,
\qquad
C_- \ = \ (2,2,1),
$$
$$
u^- \ = \ (0,-1,-\alpha_+),
\qquad
i_1^- \ = \ 1,
\qquad
C_- \ = \ (1,1,1).
$$
In both cases, we obtain 
$$
\qquad\qquad
\begin{array}{l}
\bangle{\alpha_+,d_{02}}
= 
\bangle{\alpha_+,d_{12}}
 = 
1,
\\[1ex]
\bangle{\alpha_+,d_{2j}}  =  -l_{2j}
\text{ for } j = 1,\ldots, n_2,
\\[1ex]
\bangle{\alpha_+,d_{k}'} = 0 
\text{ for } j = 1,\ldots, m.
\end{array}
$$
Now choose any invertible $s \times s$ matrix with 
$\alpha_+$ as its first row and apply it from the 
left to $P$. 
Then the third row of $P$ looks as follows
$$ 
\left[
\begin{array}{rrrrrr}
0 &  1     & 0   & 1     &  -l_2 & 0
\end{array}
\right]
$$
Adding suitable multiples of the third row to the 
last $s-1$ rows and adding the second to 
the third row brings $P$ into the desired form.
The remaining statements are directly checked. 
\end{proof}

\begin{lemma}
\label{lem:horss1more}
Consider a pair $\alpha_\pm \in \Phi_P^{\rm hor}$
with distinguished index $i_0 = 2$.
There exists another pair 
$\tilde \alpha_\pm \in \Phi_P^{\rm hor}$
with distinguished index $\tilde i_0 = 2$
if and only if we have 
$$ 
\deg_K(T_{01}) 
\ = \ 
\deg_K(T_{02})
\ = \ 
\deg_K(T_{11}) 
\ = \ 
\deg_K(T_{12}).
$$
Moreover, if the latter holds, then 
$\alpha_\pm,\tilde \alpha_\pm$ are the only 
pairs of semisimple roots with distinguished 
index $2$ and they form a root system isomorphic 
to $A_1 \oplus A_1$.
\end{lemma}

\begin{proof}
We may assume that we are in the setting of
Lemma~\ref{lem:horss1}.
Then we just have to go through the possible 
cases $\tilde i_1^\pm$ and $C_-$ and observe 
that the existence of $\tilde \alpha_\pm$
implies a special shape of $P$ equivalent to 
the above degree condition.
\end{proof}

\begin{lemma}
\label{lem:horss2}
Suppose $n_0 = 1$, $l_{01} = 2$ and $n_2 = 2$, $l_{11} = l_{12} = 1$.
Then there is at most one pair 
$\alpha_\pm \in \Phi_P^{\rm hor}$ 
with distinguished index $i_0 = 2$.
If there is one, then $P$ can be brought 
by admissible 
operations, without moving the $n_2$-block,
into the form
\begin{eqnarray}
\label{glg:phor2} 
P 
& = &
\left[
\begin{array}{rrrrr}
-2       & 1 & 1 & 0     & 0 
\\
-2       & 0 & 0 & l_2   & 0 
\\
-1       & 0 & 1 & 0     &  0
\\
d_{01}^* & 0 & 0 & d_2^* & d'_* 
\end{array}
\right],
\end{eqnarray}
where the lower line is a matrix of size 
$(s-1) \times (n+m)$.
Conversely, if $P$ is of the above shape,
then $\alpha_\pm = (\pm 1,0) \in \Phi_P^{\rm hor}$
has distinguished index $i_0 =2$.
Moreover, up to admissible operations of 
type~(III) and~(IV), situation~(\ref{glg:phor2}) 
is equivalent to 
$$ 
\deg_K(T_{01}) 
\ = \ 
\deg_K(T_{11}) 
\ = \ 
\deg_K(T_{12}).
$$
\end{lemma}

\begin{proof}
This is a similar computation as in the previous lemma.
Clearly, we may assume $C_+ = (1,1,1)$ and by suitable 
row operations, we bring $P$ into the form 
$$ 
P 
\ = \ 
\left[
\begin{array}{rrrrr}
-2       & 1 & 1 & 0     & 0 
\\
-2       & 0 & 0 & l_2   & 0 
\\
d_{01}   & 0 & d_{12} & d_2 & d' 
\end{array}
\right]
$$
Now enter the defining conditions of a Demazure 
$P$-root $\kappa_+$ with $u^+ = (u_1^+,u_2^+,\alpha_+)$.
It turns out that $C_- = (1,2,2)$ must hold.
We end up with $u_1^+ = 0$ and 
$$
\qquad\qquad
\begin{array}{l}
\bangle{\alpha_+,d_{01}}
= 
2u_2^+ -1,
\\[1ex]
\bangle{\alpha_+,d_{12}}
= 
1,
\\[1ex]
\bangle{\alpha_+,d_{2j}}  =  -u_2^+l_{2j}
\text{ for } j = 1,\ldots, n_2,
\\[1ex]
\bangle{\alpha_+,d_{k}'} = 0 
\text{ for } j = 1,\ldots, m.
\end{array}
$$
As before, this enables us to bring $P$ via suitable 
row operations into the desired form. 
Again, the remaining statements are directly seen.  
\end{proof}

\begin{proof}[Proof of Theorem~\ref{thm:sesipart}]
Lemma~\ref{lem:phivert} shows that $\Phi_P^{\rm vert}$
is a root system and has the desired form.
This proves~(ii).
Concerning $\Phi_P^{\rm hor}$, assertion~(iii) 
as well as the cases $l_{21} + \ldots + l_{2n_2} \ge 3$ 
of assertions~(iv) and~(v) are proven 
by Lemmas~\ref{lem:nlest}, \ref{lem:horss1}
and~\ref{lem:horss1more}.
The remaining cases of~(iv) and~(v) are
deduced using the special shape of $P$ 
guaranteed by Lemma~\ref{lem:nlest}.
So, $\Phi_P^{\rm vert}$ and $\Phi_P^{\rm hor}$
are root systems of the desired shape.
Lemma~\ref{lem:root2nld} and 
the decomposition of $K_0$ into a vertical 
and a horizontal part show that $\Phi_P^{\rm ss}$
splits into the direct sum of 
$\Phi_P^{\rm vert}$ and $\Phi_P^{\rm hor}$.

To conclude the proof, we have to verify
that $\Phi_P^{\rm ss}$ is in fact the root system 
of the semisimple part of $\Aut(X)$.
For this, it suffices to realize the root system
$\Phi_P^{\rm ss}$ as the root system of some
semisimple subgroup $G \subseteq \Aut(X)$
(which then necessarily is a semisimple part).
Consider $\b{X} := \Spec \, R(A,P)$ and 
the action of $H_X := \Spec \, \KK[K]$.
The group $G \subseteq \Aut(X)$ will be induced 
by a representation of a suitable semisimple 
group on $\KK^{n+m}$ commuting with 
the action of $H_X$
and leaving $\b{X} := \Spec \, R(A,P)$ 
invariant.

We first care about $\Phi_P^{\rm vert}$.
Consider the semisimple group 
$G^{\rm vert} := \times_p \, \SL_{m_p}$.
It acts on $\KK^{n+m}$: triviallly on
$\KK^n$ and blockwise on 
$\KK^m = \oplus_p \KK^{m_p}$.
This action commutes with the action of $H_X$
and leaves $\b{X}$ as well as $\rq{X}$ invariant.
Thus the  $G^{\rm vert}$-action descends to~$X$.
This realizes $G^{\rm vert} \subseteq \Aut(X)$ 
as a subgroup with root system~$\Phi_P^{\rm vert}$.

Similarly, going through the cases, we 
realize $\Phi_P^{\rm hor}$ as a root system 
of a semisimple group 
$G^{\rm hor} \subseteq \Aut(X)$.
Recall that for the defining relation, we have 
the possibilities
$$ 
(a) \quad T_{01}T_{02} + T_{11}T_{12} + T_2^{l_2},
\qquad \qquad 
(b) \quad T_{01}T_{02} + T_{11}^2 + T_2^{l_2}.
$$
Consider case (a).
If we have the equations 
$\b{w}_{01} = \b{w}_{11}$
and
$\b{w}_{02} = \b{w}_{12}$,
then $g \in \SL_2$ acts on the 
$T_{01},T_{11}$-space as
$(g^t)^{-1}$ and on the 
$T_{02},T_{12}$-space as $g$. 
By trivial extension, we obtain an action on 
$\KK^{n+m}$ commuting with the $H$-action,
leaving the defining relation and thus 
$\b{X}$ invariant.
Thus, we have an induced effective action of a 
semisimple group $G^{\rm hor}$ 
with root system $A_1$ on~$X$.

If we have
$\b{w}_{01} = \b{w}_{02} = \b{w}_{11} = \b{w}_{12}$,
then the canonical action of the group 
$\SO_4$ on the $T_{01},T_{11},T_{02},T_{12}$-space
extends trivially to $\KK^{n+m}$,
commutes with the action of $H_X$ and 
leaves $\b{X}$ invariant.
Again, this gives an induced effective 
action of a semisimple group $G^{\rm hor}$ 
on $X$, this time the root system 
is $A_1 \oplus A_1$; recall that $\SO_4$ 
has $\SL_2 \times \SL_2$ as its universal 
covering.

Now let $n_2=2$ and $l_{21} = l_{22}=1$.
Then we have $n=6$.
If all six $K$-degrees $\b{w}_{ij}$ coincide,
then take the action of $\SO_6$ on $\KK^n$ 
and extend it trivially to $\KK^{n+m}$.
This induces an action of a semisimple group 
$G^{\rm hor}$ on $X$.
The root system is that of the universal
covering $\SL_4$, i.e. we obtain $A_3$.
If all $\b{w}_{i1}$ and all $\b{w}_{i2}$ 
coincide but we have $\b{w}_{01} \ne \b{w}_{02}$, 
then consider the action of $\SL_3$ 
given by $(g^t)^{-1}$ on the $T_{i1}$-space
and  by $g$ on the $T_{i2}$-space.
This leads to a $G^{\rm hor}$ with root 
system~$A_2$.

Finally, consider case~(b).
If $\b{w}_{01} = \b{w}_{02} = \b{w}_{11}$ holds, 
then the canonical action of the group 
$\SO_3$ on the $T_{01},T_{02},T_{11}$-space 
defines a semisimple subgroup $G^{\rm hor}$ 
of $\Aut(X)$ with root system $A_1$.
If $n_2 = 1$ and $l_{21} = 2$ holds
and we have 
$\b{w}_{01}=\b{w}_{02}=\b{w}_{11} = \b{w}_{21}$,
then $n=4$ holds and the canonical action 
of $\SO_4$ on $\KK^n$ induces a semisimple 
subgroup $G^{\rm hor}$ of $\Aut(X)$ with 
root system $A_1 \oplus A_1$.
If we have $n_2 = 2$ and $l_{21} = l_{22} = 1$
and all degrees $\b{w}_{ij}$ coincide,
then the canonical action of $\SO_5$ on 
$\KK^5$ extends to $\KK^{n+m}$ and 
induces a subgroup $G^{\rm hor}$ 
of $\Aut(X)$ with root system~$B_2$.

Alltogether, we realized the root systems 
$\Phi_P^{\rm vert}$ and $\Phi_P^{\rm hor}$ by 
semisimple subgroups $G^{\rm vert}$ and 
$G^{\rm hor}$ of $\Aut(X)$.
By construction, these groups commute and 
thus define a semisimple subgroup 
$G :=  G^{\rm vert}G^{\rm hor}$ of $\Aut(X)$ 
with the desired root system 
$\Phi_P^{\rm ss} = \Phi_P^{\rm vert} \oplus \Phi_P^{\rm hor}$.
\end{proof}

\begin{corollary}[of proof]
In the situation of Theorem~\ref{thm:sesipart},
any pair of semisimple roots
defines a subgroup of the automorphism 
group locally isomorphic to $\SL_2$.
\end{corollary}

\section{Applications}
\label{sec:appsem}

In this section we present applications of 
Theorem~\ref{thm:sesipart}. 
A first one concerns the automorphism group of arbitrary 
nontoric Mori dream surfaces (not necessarily admitting 
a $\KK^*$-action).

\begin{proposition}
\label{prop:alhomsurf}
Let $X$ be a nontoric Mori dream surface.
Then $\Aut(X)^0$ is solvable and the following 
cases can occur:
\begin{enumerate}
\item
$X$ is a $\KK^*$-surface,
\item
$\Aut(X)^0$ is unipotent.
\end{enumerate}
\end{proposition}

\begin{proof}
Consider a maximal torus $T \subseteq \Aut(X)^0$.
If $T$ is trivial, then we are in case~(ii).
In particular, $\Aut(X)^0$ is solvable then.
If $T$ is one-dimensional, then we are  
in case~(i) and the task is to show that the 
semisimple part of $\Aut(X)$ is trivial, i.e. 
$\Aut(X)$ has no semisimple roots.
For this, we remark first that as a Mori dream 
surface with a nontrivial $\KK^*$-action, 
$X$ is rational.
Thus, we may assume that $X$ arises from 
$A,P$ as in Construction~\ref{constr:RAPdown} 
and that $R(A,P)$ is minimally presented.
Note that we have $s=1$.

The estimate of Theorem~\ref{thm:sesipart}~(ii)
forbids vertical semisimple $P$-roots.
Let us see why there are no horizontal 
semisimple $P$-roots. Otherwise, Lemmas~\ref{lem:horss1} 
and~\ref{lem:horss2} show that we must have $m=0$ 
and, up to admissible operations, 
the matrix $P$ is one of the following:
$$ 
\left[
\begin{array}{rrrrr}
-1 & -1 & 1 & 1 & 0 
\\
-1 & -1 & 0 & 0 & l_2 
\\
-1 &  0     & 0   & 1     &  0 
\end{array}
\right],
\qquad\qquad
\left[
\begin{array}{rrrr}
-2       & 1 & 1 & 0     
\\
-2       & 0 & 0 & l_2   
\\
-1       & 0 & 1 & 0     
\\
\end{array}
\right].
$$
Since the columns of $P$ are pairwise different 
primitive vectors, we obtain $n_2 = 1$ and 
$l_{21} = 1$;
a contradiction to the assumption that the ring 
$R(A,P)$ is minimally presented.
\end{proof}

We take a brief look at $q$-dimensional 
varieties~$X$ coming with a nontrivial 
action of $\SL_q$.
They were classified (in the smooth case)  
by Mabuchi~\cite{Ma2}.
Let us see how to obtain the rational 
nontoric part of his list with the aid 
of Theorem~\ref{thm:sesipart}.

\begin{example}
Let $X$ a nontoric $q$-dimensional complete normal variety 
with a nontrivial action of~$\SL_q$.
Since $\SL_q$ has only finite normal subgroups, $X$ comes 
with a torus action $T \times X \to X$ of complexity one.
Proposition~\ref{prop:alhomsurf} tells us that 
$X$ is of dimension at least three.
If $\SL_q$ acts with an open orbit, then $X$ has a 
finitely generated divisor class group
and, by existence of the $T$-action of 
complexity one, must be rational. 
So, we may assume that we are in the setting 
of Theorem~\ref{thm:sesipart}, 
where we end up in the cases $A_2,A_3$ of second 
item of~(iv) which amounts to the following two
possibilities.
\begin{enumerate}
\item 
We have $q=3$ and $X$ is the flag variety $\SL_3/B_3$;
in particular, $\Cl(X) \cong \ZZ^2$ and 
$$
\qquad\qquad
\mathcal{R}(X)
\ \cong \ 
\KK[T_{01},T_{02},T_{11},T_{12},T_{21},T_{22}]
\ / \ \bangle{T_{01}T_{02}+T_{11}T_{12}+T_{21}T_{22}}
$$
hold, where the $\ZZ^2$-grading is given by 
$$
\deg(T_{01}) \ = \ \deg(T_{11}) \ = \ \deg(T_{21}) \ = \ (1,0),
$$
$$
\deg(T_{02}) \ = \ \deg(T_{12}) \ = \ \deg(T_{22}) \ = \ (0,1).
$$
\item 
We have $q=4$ and $X$ is the smooth quadric
$V(T_{0}T_{1}+T_{2}T_{3}+T_{4}T_{5})$ in $\PP_5$,
where $\SL_4$ acts as the universal covering of $\SO_6$.
\end{enumerate}
Now assume that $\SL_q$ acts with orbits of 
dimension at most $q-1$. 
Then this action gives a root system 
$A_{q-1} \subseteq \Phi_P^{\rm vert}$.
If $X$ were rational, then Theorem~\ref{thm:sesipart}~(ii) 
would require $q < \dim(X)$ which is excluded by assumption. 
Thus, $X$ must be nonrational.
In particular, the  rational nontoric part of Mabuchi's list is
established, even for a priori singular varieties.
\end{example}

We turn to nontoric varieties with a torus action 
of complexity one which are almost homogeneous under 
some reductive group.
Recall that an action of a reductive group $G$ on
$X$ is spherical if some Borel subgroup of
$G$ acts with an open orbit on $X$.

\begin{proposition}
\label{prop:sphvshoss}
Let $X$ be a nontoric complete normal variety. 
Then the following statements are equivalent.
\begin{enumerate}
\item
$X$ allows a torus action of complexity one and 
an almost homogeneous reductive group action.
\item
$X$ is spherical with respect to an action of
a reductive group of semisimple rank one.
\item
$X$ is isomorphic to a variety as in 
Theorem~\ref{thm:sesipart}~(iii).  
\end{enumerate}
\end{proposition}

\begin{proof}
The implication ``(ii)$\Rightarrow$(i)'' is obvious.
For the implication ``(iii)$\Rightarrow$(ii)'', take 
a pair of semisimple horizontal $P$-roots $\alpha_{+}$
and $\alpha_{-}$.
Then the acting torus $T$ of $X$ and the 
root subgroups $U_{\pm}$ associated 
to $\alpha_{\pm}$ generate a reductive group
$G$ of semisimple rank one in $\Aut(X)$ and $X$
is spherical with respect to the action of $G$.

Assume that (i) holds. 
Then the open orbit of the acting reductive group $G$
is unirational.
Thus, $X$ is unirational. 
By the existence of a torus action of complexity one, 
$X$ contains an open subset of the form 
$T \times C$ with some affine curve $C$. 
We conclude that $C$ and hence $X$ are rational. 
Consequently, $X$ is a Mori dream space.
In particular, Corollary~\ref{cor:autmds} yields
that $\Aut(X)$ is linear algebraic.
Moreover, we may assume that we are in 
the setting of Theorem~\ref{thm:sesipart}.
The image of $G^{0}$ in $\Aut(X)$ is contained in a 
maximal connected reductive subgroup $G'$ of $\Aut(X)$.
Suitably conjugating $G'$, we may assume that the 
acting torus $T$ of $X$ is a maximal torus of $G'$.
Since $G'$ is generated by root subgroups,
we infer from Corollary~\ref{cor:autXgen} that  
there must be a horizontal Demazure $P$-root.
Since every root of $G'$ is semisimple, we  
end up in Case~\ref{thm:sesipart}~(iii).
\end{proof}

Specializing to dimension three, we obtain a quite
precise picture of the possible matrices $P$ in 
the above setting.

\begin{proposition}
\label{prop:dim3alhomred}
Let $X$ be a three-dimensional nontoric complete 
normal rational variety.
Suppose that $X$ is almost homogeneous under 
an action of a reductive group and there 
is an effective action of a two-dimensional 
torus on $X$.
Then the Cox ring of $X$ is given 
as $\mathcal{R}(X) = R(A,P)$ with
a matrix $P$ according to the following 
cases

\begin{enumerate}
\item
\qquad
$ 
P \quad = \quad 
\left[
\begin{array}{rrrrrr}
-1 & -1 & 1 & 1 & 0 & 0 
\\
-1 & -1 & 0 & 0 & l_2 & 0 
\\
-1 &  0     & 0   & 1     &  0 & 0
\\
0  &  0 & 0  & d_{12}^* & d_2^* & d'_* 
\end{array}
\right]
$,
\\[.5em]
\item
\qquad
$
P \quad = \quad
\left[
\begin{array}{rrrrr}
-2       & 1 & 1 & 0     & 0 
\\
-2       & 0 & 0 & l_2   & 0 
\\
-1       & 0 & 1 & 0     &  0
\\
d_{01}^* & 0 & 0 & d_2^* & d'_* 
\end{array}
\right]
$.
\end{enumerate}

\noindent
In both cases, we have $m \le 2$; that means that 
the $d'_*$-part can be either empty, equal to 
$\pm 1$ or equal to $(\pm 1,\mp 1)$.
\end{proposition}

\begin{proof}
Clearly, we may assume that we are in the situation 
of Theorem~\ref{thm:sesipart}.
Since $X$ is nontoric but almost homogeneous,
there must be a semisimple horizontal $P$-root.
Thus, Lemmas~\ref{lem:horss1} and~\ref{lem:horss2}
show that after admissible operations, $P$ is 
of the desired shape.
\end{proof}

As a direct consequence, one retrieves results
of Haddad~\cite{Had} on the Cox rings
of  three-dimensional varieties that are almost 
homogeneous under an $\SL_2$-action and additionally 
come with an effective action of a two-dimensional 
torus $T$; compare also~\cite{AL} for the affine 
case.

\begin{corollary}
Let $X$ be a three-dimensional 
nontoric complete normal rational 
variety.
Suppose that $X$ is almost homogeneous 
under an action of $\SL_2$
and there is an effective action 
of a two-dimensional torus $T$ on $X$.
Then $X$ has at most two 
$T$-invariant prime 
divisors with infinite $T$-isotropy,
i.e., we have $m \le 2$ and 
the Cox ring of $X$ is given as
\begin{eqnarray*} 
\mathcal{R}(X)
& = & 
\KK[T_{ij},S_k] / 
\bangle{T_{01}T_{02}+T_{11}T_{12}+T_{21}^{l_{21}}\cdots T_{2n_2}^{l_{2n_2}}}.
\end{eqnarray*}
\end{corollary}

\begin{proof}
We are in the situation of Proposition~\ref{prop:dim3alhomred}. 
The assumption that $\SL_2$ acts with an open orbit implies
that we are in case~(i).
\end{proof}

Finally, we consider almost homogeneous 
varieties with reductive automorphism
group; see for example~\cite{Nill} for 
results on the toric case.
Here, we list all almost homogeneous  
threefolds of Picard number one with 
a reductive automorphism group having 
a maximal torus of dimension two.

\begin{proposition}
\label{prop:3dimautred}
Let $X$ be a $\QQ$-factorial three-dimensional 
complete normal variety of Picard number one. 
Suppose that $\Aut(X)$ is reductive, has a 
maximal torus of dimension two and acts
with an open orbit on $X$.
Then $X$ is a rational Fano variety and,
up to isomorphy, arises from a matrix $P$ 
from  the following list:
\begin{enumerate}
\item
$
\left[
\begin{array}{rrrrr}
-1 & -1 & 1 & 1 & 0 
\\
-1 & -1 & 0 & 0 & l_{21} 
\\
-1 &  0     & 0   & 1     &  0 
\\
0  &  0 & 0  & d_{12} & d_{21}
\end{array}
\right],
\enspace
l_{21} > 1, \
d_{12} > 2, \
- \frac{d_{21}}{d_{12}-1} < l_{21} <  - d_{21},
$
\\[.5em]
\item
$
\left[
\begin{array}{rrrrr}
-2 & 1 & 1 & 0 & 0 
\\
-2 & 0  & 0 & l_{21} & l_{22} 
\\
-1 & 0  & 1  & 0     &  0 
\\
d_{01}  & 0  & 0 & d_{21} & d_{22}
\end{array}
\right],
\enspace
l_{21}, l_{22} >  1, 
\ 
2d_{22} > -d_{01}l_{22},
\
-2d_{21} > d_{01}l_{21},
$
\\[.5em]
\item
$
\left[
\begin{array}{rrrrr}
-2 & 1 & 1 & 0 & 0 
\\
-2 & 0  & 0 & 1 & l_{22} 
\\
-1 & 0  & 1  & 0     &  0 
\\
d_{01} & 0  & 0 & d_{21} & d_{22}
\end{array}
\right],
\begin{array}{l}
l_{22} >  1, \
d_{22} > d_{21}l_{22} + l_{22}, \
2d_{22} > -d_{01}l_{22}, \
-2d_{21} > d_{01},
\\
\text{or }
\\
l_{22} >  1, \
2d_{22} > -d_{01}l_{22}, \
1-2d_{21} > d_{01},
\end{array}
$
\\[.5em]
\item
$
\left[
\begin{array}{rrrrr}
-2 & 1 & 1 & 0 & 0 
\\
-2 & 0  & 0 & 1 & 1
\\
-1 & 0  & 1  & 0     &  0 
\\
-1  & 0  & 0 & 1 & 0
\end{array}
\right],
$
\\[.5em]
\item
$
\left[
\begin{array}{rrrrr}
-2 & 1 & 1 & 0 & 0 
\\
-2 & 0  & 0 & l_{21} & 0
\\
-1 & 0  & 1  & 0     &  0 
\\
1  & 0  & 0 & d_{21} & 1
\end{array}
\right],
\enspace
1 < l_{21} < -2d_{21} < 2l_{21}.
$
\end{enumerate}
Conversely, each of the above listed matrices defines a
$\QQ$-factorial rational almost homogeneous Fano variety 
with reductive automorphism group having a two-dimen\-sional 
maximal torus.
\end{proposition}

\begin{proof}
According to Proposition~\ref{prop:sphvshoss},
we may assume that
$X$ arises from a matrix $P$.
The fact that $\Aut(X)$ is reductive means that 
every root must be semisimple and
the fact that it acts with an open orbit
means that there exists at least one pair of 
horizontal semisimple roots.
Now, suppose we have $m=0$. 
If there is only one pair of 
horizontal semisimple $P$-roots, 
then reductivity of the automorphism group 
forbids further $P$-roots and we 
end up with the first three cases. 
If there are more then one pair of semisimple 
roots, then we end up in case four.
Finally, if $m > 0$ holds, then $m=1$ is the only 
possibility and one is left with case~(v).
\end{proof}

\end{document}